\documentclass[11pt]{article}

 \usepackage{amsmath,amssymb,amscd,amsthm,esint}
 
\usepackage{graphics,amsmath,amssymb,amsthm,mathrsfs}

\usepackage{graphics,amsmath,amssymb,amsthm,mathrsfs,amsfonts}

\usepackage{sidecap}
\usepackage{float}
\usepackage{extarrows}
\usepackage{booktabs}
\usepackage{verbatim}
\usepackage{hyperref}
\usepackage[usenames,dvipsnames]{xcolor}

\newtheorem{thm}{Theorem}[section]
\newtheorem{lemma}[thm]{Lemma}

\theoremstyle{definition}
\newtheorem{remark}[thm]{Remark}

\def\XXint#1#2#3{{\setbox0=\hbox{$#1{#2#3}{\int}$}
         \vcenter{\hbox{$#2#3$}}\kern-.5\wd0}}

\def\R{\mathbb{R}}

\def\e{\varepsilon}
\def\wN{N^*}
\def\Sd{\mathbb{S}^{d-1}}

\numberwithin{equation}{section}

\begin{document}

\title{Critical Sets of Solutions of Elliptic Equations \\ in Periodic Homogenization}

\author{Fanghua Lin \thanks{Supported in part by NSF grant DMS-1955249.}
\qquad
Zhongwei Shen\thanks{Supported in part by NSF grant DMS-1856235 and by Simons Fellowship}}
\date{}
\maketitle
\begin{abstract}

In this paper we study critical sets of solutions $u_\e$  of second-order elliptic equations in divergence form with rapidly oscillating and periodic coefficients.
We show that the $(d-2)$-dimensional Hausdorff  measures of the critical sets are bounded uniformly with respect to the period $\e$, 
provided that  doubling indices for solutions are bounded.
 The key step is an estimate of "turning"  for the projection of a non-constant solution $u_\e$ 
 onto the subspace of spherical harmonics of order $\ell$, when the doubling index for $u_\e$ on a
 sphere $\partial B(0, r)$ is trapped between $\ell -\delta$ and $\ell +\delta$, for $r$ between $1$ and a minimal radius $r^*\ge C_0\e$.
This estimate is proved by using  harmonic approximation successively.
With a suitable $L^2$ renormalization as well as rescaling  we are able to control the accumulated errors introduced  by homogenization and projection.
Our proof also gives uniform bounds for  Minkowski contents of the critical sets.

\medskip

\noindent{\it Keywords}: Critical Set;  Homogenization; Hausdorff Measure; Doubling Index.

\medskip

\noindent {\it MR (2010) Subject Classification}: 35J15, 35B27.

\end{abstract}

\section{Introduction}

In this paper we initiate the study of critical points of solutions of elliptic equations in homogenization.
More precisely, we consider a family of second-order elliptic operators in divergence form,
\begin{equation}\label{op}
\mathcal{L}_\e =-\text{\rm div} (A(x/\e) \nabla ),
\end{equation}
where $0< \e\le 1$ and  $A(y) = (a_{ij} (y))$ is a  $d\times d$ matrix-valued function in $\R^d$.
Let $u_\e$ be a non-constant  weak solution of $\mathcal{L}_\e (u_\e)=0$ in a unit ball $B(0, 1)$ and
\begin{equation}\label{critical-0}
\mathcal{C} (u_\e) =\big\{ x: \ |\nabla u_\e (x)|=0 \big\},
\end{equation}
the critical set of $u_\e$.
Under some  smoothness assumptions on the coefficient matrix,
it is known that  the critical set is $(d-2)$-dimensional and its $(d-2)$-dimensional Hausdorff measure  is finite.
Moreover, in the case $\e=1$,  it was proved in \cite{Han-Lin-H-1998} that
\begin{equation}\label{critical-01}
\mathcal{H}^{d-2} (\mathcal{C} (u)\cap B(0, 1/2)) \le C (N),
\end{equation}
where $\mathcal{L}_1 (u)=0$ in $B(0, 1)$, $N$ is the Almgren frequency of $u-u(0)$ on $B(0, 1)$, and
$C(N)$ is a constant that depends on $N$.
The main purpose of this paper is to establish the Hausdorff measure estimate \eqref{critical-01} 
for solutions of $\mathcal{L}_\e (u_\e)=0$, with $C(N)$ {\it independent } of $\e$, under the assumption that $A$ is periodic.

Throughout the paper, unless indicated otherwise, we shall assume that 
\begin{itemize}

\item

(ellipticity) there exists some $\lambda\in (0, 1]$ such that
\begin{equation}\label{ellipticity}
\lambda |\xi|^2 \le \langle A(y) \xi, \xi \rangle\quad \text{ and } \quad 
 |\langle A(y)\xi, \zeta\rangle | \le \lambda^{-1} |\xi| |\zeta| \quad \text{ for any } y, \xi, \zeta\in \R^d;
\end{equation}

\item 

(periodicity) $A$ is periodic with respect to some lattice $\Gamma$ of $\R^d$,
\begin{equation}\label{periodicity}
A(y+z)  =A (y) \quad \text{ for any } y\in \R^d \text{ and } z\in \Gamma;
\end{equation}

\item

(smoothness) there exists some $M>0$ such that
\begin{equation}\label{smoothness}
|A(x)-A(y)| \le M |x-y| \quad \text{ for any } x, y \in \R^d.
\end{equation}
\end{itemize}
Let $\chi (y)=(\chi_j (y))$ denote the first-order corrector for $\mathcal{L}_\e$.
We will also assume that the periodic matrix $I +\nabla \chi$ is nonsingular  and that
\begin{equation}\label{inv-0}
\text{det} (I+\nabla \chi)  \ge  \mu
\end{equation}
for some $\mu>0$.
The condition \eqref{inv-0} holds in the case $d=2$ if $A$ is periodic and H\"older continuous.
In this paper we assume $d\ge 3$; the two-dimensional case is treated separately in  \cite{Two-D},
using a different approach.
Note that the condition \eqref{inv-0} is satisfied if $\|\text{\rm div}(A)\|_\infty < c_0$
for some small constant $c_0$ depending on $d$, $\lambda$, $\Gamma$ and $M$.
Since $u_\e = x_j +\e \chi_j (x/\e)$ is a solution  for $j=1, \dots , d$, the  condition
  is  also necessary for the Minkowski estimate \eqref{H-M-00} below with $r>\e^{d/(d-2)}$.

The following is the main result of this paper.

\begin{thm}\label{main-thm}
Suppose  that $A=A(y)$ satisfies the conditions \eqref{ellipticity}, \eqref{periodicity},  \eqref{smoothness} and \eqref{inv-0}.
Let $u_\e \in H^1(E_2)$ be a non-constant weak solution of $\mathcal{L}_\e (u_\e)=0$ in $E_2\subset\R^d$, $d\ge 3$.
Suppose that $u_\e (0)=0$ and
\begin{equation}\label{doubling-0}
\fint_{E_2} u_\e^2 \le 4^N  \fint_{E_1} u_\e^2
\end{equation}
for some $N>1$. Then
\begin{equation}\label{H-M-00}
|  \big\{ x : \text{\rm dist} (x, \mathcal{C}(u_\e)\cap E_{1/2}  )< r \big\}| 
\le C(N) r^2
\end{equation}
for $0< r< 1$, 
and consequently, 
\begin{equation}\label{H-M-0}
\mathcal{H}^{d-2} \big\{ x\in E_{1/2}: \ |\nabla u_\e (x)|=0 \big\}
\le C(N),
\end{equation}
where $C(N)$ depends at most on $d$, $\lambda$, $\Gamma$, $M$, $\mu$,  and $N$.
\end{thm}

In Theorem \ref{main-thm} we have used the notation 
\begin{equation}
E_r = \big\{ x \in \R^d: \ 2\langle (\widehat{A} + (\widehat{A})^T)^{-1} x,  x\rangle <  r^2 \big\}
\end{equation}
for $r>0$,
where $\widehat{A}$ denotes the homogenized matrix for $A$.
Notice that $E_r=B(0, r)$ if $\widehat{A}+(\widehat{A})^T=2I$.

Classical results in the study of nodal, singular, and critical sets for solutions and eigenfunctions of elliptic operators
may be found in \cite{D-Fefferman, Hardt-Simon, Lin-N, Han-1994,  Han-Lin-2, Han-Lin-H-1998, HHN-1999}.
We refer the reader to \cite{Naber-2015, Naber-2017, BET-2017,  Logunov-2018-1, Logunov-2018-2} and their references for more recent work in this area.
In particular, for critical sets in the case $\e=1$, the local finiteness of the $(d-2)$-dimensional Hausdorff measure  was established
in \cite{HHN-1999}, and the measure estimate \eqref{critical-01} was proved in \cite{Han-Lin-H-1998}.
The results in \cite{HHN-1999, Han-Lin-H-1998} were proved under the assumption that the coefficients are sufficiently smooth.
Minkowski  estimates for nodal and critical sets, similar to \eqref{H-M-00},  were obtained in \cite{Naber-2017} under the
Lipschitz condition \eqref{smoothness}.
The smoothness condition \eqref{smoothness} is more or less sharp, as the unique continuation property fails
for operators with H\"older continuous coefficients.

The quantitative results for nodal and critical sets in the references mentioned above do not extend directly to the operator $\mathcal{L}_\e$, 
since the  bounding constants $C(N)$ depend on the smoothness of coefficients.
In \cite{LS-Nodal} we studied the nodal set, 
\begin{equation}\label{Nodal-0}
\mathcal{N}(u_\e) =\big\{ x: u_\e (x)=0 \big\}
\end{equation}
for solutions of $\mathcal{L}_\e (u_\e)=0$ in $E_2$.  Under the conditions \eqref{ellipticity}, \eqref{periodicity} and \eqref{smoothness},
we were able to show that
\begin{equation}\label{Nodal-2}
\mathcal{H}^{d-1} \big\{ x\in E_{1/2}: \ u_\e (x)=0 \big\}
\le C(N),
\end{equation}
where $C(N)$ depends at most on $d$, $\lambda$, $\Gamma$,  $M$, and $N$.
The question of explicit dependence of $C(N)$ on $N$ in a doubling inequality, which plays a key role in the proof of \eqref{Nodal-2},
 was subsequently addressed in \cite{Kenig-Zhu-2021, Kenig-ZZ-2021, Armstrong-2021}.
This paper continues the study of geometric properties for the operator $\mathcal{L}_\e$, arising in the theory of homogenization.
As in the case of $\mathcal{L}_1$, the critical sets are much harder to handle than  the nodal sets.
Theorem \ref{main-thm} contains the first result on geometric measure estimates, that are uniform in $\e>0$,  for critical sets of solutions of
$\mathcal{L}_\e (u_\e)=0$.
The result is new even for smooth coefficients.

We now describe our approach to Theorem \ref{main-thm}.
Let $u_\e$ be a non-constant solution of $\mathcal{L}_\e (u_\e)=0$ in a ball $B(0, 2)$.
We  introduce a doubling index for $u_\e$,
\begin{equation}\label{double-0}
\wN(u_\e, x_0, r)=\log_4
\frac{\fint_{\partial B(x_0, r)} (u_\e -u_\e (x_0))^2}{\fint_{\partial B(x_0, r/2)} (u_\e-u_\e (x_0))^2},
\end{equation}
using spheres instead of balls.
The doubling index, which  is nondecreasing in $r$ for harmonic functions, may not be monotone  for solutions of $\mathcal{L}_\e (u_\e)=0$.
However, a substitute for the monotonicity, which often plays a crucial  role in the study of geometric properties of solutions of elliptic and parabolic equations,
may be obtained by a compactness argument,  as in \cite{LS-Nodal}.
This gives us control for the upper bounds of $\wN(u_\e, x_0, r)$ for small $r$.
See Theorem \ref{d1-thm}.

Next, we  fix an integer $\ell\ge 2$ and $\delta_0 \in (0, 1/2)$ sufficiently small.
Assume that $\wN(u_\e, 0, 1) \le \ell +\delta_0$.
We introduce a minimal radius, 
\begin{equation}\label{rad-1}
r^*= \sup \big\{  \e_0^{-1} \e  < r\le  (1/2): \ \wN (u_\e, 0, r) \le \ell -\delta_0 \big\},
\end{equation}
where $\e_0>0$ is small.
Note that for $r^*< r< (1/2) $, we have 
\begin{equation}\label{rad-2}
\wN (u_\e, 0, r) \in [ \ell-\delta_0, \ell +\delta_0].
\end{equation}
Suppose $0<r^* <(1/4)$.
One of the key  estimates in this paper is the following estimate of "turning" from the scale $(1/4)$ to scale $r$,
\begin{equation}\label{rad-3}
\Big\|
\frac{P_\ell (u_\e (r, \cdot))}{ \| u_\e (r, \cdot)\| }
-\frac{P_\ell (u_\e (1/4, \cdot))}{ \| u_\e (1/4, \cdot) \|} \Big\|
\le C \left\{  \delta_0  + (\e/r)^{1/2} \right\},
\end{equation}
for $r^*\le  r < (1/4)$ and $0< \e< \e_0$,
where $\| \cdot \|$ denotes the norm in $L^2(\Sd)$ and $C$ depends on $\ell$.
In \eqref{rad-3} we have used the notation $u_\e (r, \omega) =u_\e (r\omega)$ for $r>0$ and $\omega\in \Sd$.
Also, for each fixed $r$,
$P_\ell (u_\e (r, \cdot))$ denotes the homogeneous harmonic polynomial of
degree $\ell$ whose restriction on $\Sd$ is the projection of $u_\e (r, \cdot)$ onto the subspace of spherical harmonics of order $\ell$.
Using \eqref{rad-2} and harmonic approximation, one may show that $u_\e (r, \omega)$ is dominated by its projection $P_\ell (u_\e (r, \cdot))$ in the sense that
\begin{equation}\label{rad-4}
\| u_\e  (r, \cdot) - P_\ell (u_\e (r, \cdot))\|
\le C (\sqrt{\delta_0} +\sqrt{\e/r}) \| u_\e (r, \cdot)\|.
\end{equation}
This, together with \eqref{rad-3}, allows us to control $|\nabla u_\e (r, \omega)|$ from below by
$|\nabla u_\e (1/4, \omega)|$, as long as $r^*\le r< 1/4$.
In particular, it  follows  that if  an approximate tangent plane $V$ at $u_\e (1/4, \cdot)$ is of dimension $d-2$, the solution 
$u_\e$ has no critical point in the region, 
\begin{equation}\label{rad-4-0}
\big\{ x\in B(0, c_0) \setminus B(0, r^*): \ \text{dist}(x, V) \ge \gamma |x| \big\},
\end{equation}
as in the case of Laplace operator.
We mention that the assumption on the invertibility of $I+\nabla \chi$ is used in the approximation of
$\nabla u_\e$ by  $(I+\nabla \chi (x/\e) )\nabla u_0$, where $u_0$ is a harmonic approximation of $u_\e$.
Also, the appearance $\e_0^{-1}\e$ in \eqref{rad-1} is natural, as the estimates for scales below $\e$ are local and 
can be obtained by a blowup argument.

The main estimate \eqref{rad-3} is proved by using harmonic approximation successively.
To carry this out, we consider the turning from the scale $2^{-j}$ to $2^{-j-1}$,
\begin{equation}\label{rad-4-1}
\Big\|
\frac{P_\ell (u_\e (2^{-j}, \cdot))}{ \| u_\e (2^{-j} , \cdot)\| }
-\frac{P_\ell (u_\e, 2^{-j-1} , \cdot))}{ \| u_\e  (2^{-j-1} , \cdot) \|} \Big\|,
\end{equation}
where $r^*< 2^{-j} < 1$.
Using the harmonic approximation, rescaling,  and a suitable renormalization, we are able to show that \eqref{rad-4-1}
is bounded by
\begin{equation}\label{rad-5}
8 \big\{ \wN(u_\e, 0, 2^{-j+1}) -\wN(u_\e, 0, 2^{-j-1}) \big\} + C (2^j \e)^{1/2}.
\end{equation}
The  second term in \eqref{rad-5}  is the rescaled harmonic approximation error ($u_\e$ by $u_0$), while
 the first is the error introduced by the approximation of $u_0$ by its projection $P_\ell (u_0)$.
Since the sum in the index $j$ of the first terms in \eqref{rad-5} is telescope,
the  inequality  \eqref{rad-3} follows from the estimates for \eqref{rad-4-1} by summation.
In the argument above we have assumed that  solutions of the homogenized equation $\mathcal{L}_0 (u_0)=0$
are harmonic. The general case may be reduced to this special case by a change of variables. See Remark \ref{symm-re}.

With the estimate \eqref{rad-3} at our disposal, we adapt  some arguments from \cite{Naber-2017} 
(also see \cite{Mattila, Simon-1995, Simon-1996, Lin-1999, Lin-Wang}  for  related classical techniques developed for harmonic maps and minimal surfaces),  and  study 
properties of two homogeneous harmonic 
 polynomials  $P_\ell (u_\e (x_0 +\cdot/4))$ and $P_\ell (u_\e (x_1+\cdot/4))$,
  assuming that  the minimal radii $r_0^*$ and $r_1^*$ for $x_0$ and $x_1$ respectively, defined as in  \eqref{rad-1}, are small relative to $|x_0-x_1|$.
  In particular, we  show that
  \begin{equation}\label{rad-6}
  \| n\cdot \nabla \psi_{\ell, x_0} \|
  \le C \big\{ \delta_0^{1/2} + (\e/ r_0^*)^{1/2} + (\e/r^*_1)^{1/2}\big\},
  \end{equation}
  where $n=(x_1-x_0)/|x_1-x_0|$, 
  $$
  \psi_{\ell, x_0} = P_\ell (u_\e (x_0 + \cdot/4))/\| P_\ell (u_\e (x_0+\cdot/4))\|,
  $$
  and $C$ depends on $\ell$.
  Finally,  Theorem \ref{main-thm} follows from \eqref{rad-6} as well as the absence of critical points in the set \eqref{rad-4-0}
   by an induction on $\ell$ and a covering argument.
   Indeed, by choosing $\e_0$ in \eqref{rad-1} sufficiently small,
   one may use \eqref{rad-6} to show that if $\{ B(x_i, r^*(x_i)/20) \}$ is a collection of disjoint balls, then the centers $\{x_i\}$
   lie on the graph of a Lipschitz function from $V$ to $V^\perp$,
   where the approximate tangent plane $V$ is a subspace of dimension $d-2$ or less.
   In a forthcoming paper \cite{Lin-Shen-2022}, based on the estimate \eqref{rad-3},
   a different argument will be used to bound the Hausdorff measures of critical sets of solutions to elliptic equations.
   
The paper is organized as follows.
In Section \ref{section-2} we collect several  results and estimates,
 which  are needed in the subsequent sections, from the theory of elliptic homogenization.
 In Section \ref{section-3} we introduce the doubling index \eqref{double-0} and establish a doubling inequality for solutions of
 $\mathcal{L}_\e (u_\e)=0$  by a compactness argument.
 In Section \ref{section-4} we study the properties of projection $P_\ell(u_\e (r, \cdot))$ and prove the key estimate  for \eqref{rad-4-1}
 mentioned earlier. This is done first for harmonic functions, using Weiss type monotonicity formulas \cite{Weiss, Garofalo-2009},  and then for $u_\e$ 
 by using harmonic approximation.
  Sections \ref{section-5} and \ref{section-6} are devoted to  the proof  of  \eqref{rad-6}.
  Finally, the proof of Theorem \ref{main-thm} is given in Section \ref{section-7}. 
 
Throughout the paper we will use $C$ and $c$ to denote constants that may depend on $d$, 
 $\lambda$ in \eqref{ellipticity}, $\Gamma$ in \eqref{periodicity},  $M$
in \eqref{smoothness}, and $\mu$ in \eqref{inv-0}.
If a constant also depends on other parameters, such as the doubling index of a solution, it will be stated explicitly.


\section{Homogenization and harmonic approximation}\label{section-2}

Let $\mathcal{L}_\e$ be the second-order elliptic operator given by \eqref{op}, where $A=A(y)$ satisfies \eqref{ellipticity} and \eqref{periodicity}.
To introduce the homogenized operator, one solves the cell problem for the  first-order correctors $\chi =\chi(y) =(\chi_j(y) )$,
\begin{equation}\label{cell-1}
\left\{
\aligned
& \mathcal{L}_1 (\chi_j) =-\frac{\partial}{\partial y_i} \big( a_{ij} \big) \quad \text{ in } Y,\\
& \fint_Y \chi_j=0 \quad \text{ and } \quad
\chi_j \text{ is Y-periodic},
\endaligned
\right.
\end{equation}
for $1\le j\le d$ (the index $i$ is summed from $1$ to $d$),
where $Y$ is the fundamental domain for the lattice $\Gamma$.
The homogenized operator $\mathcal{L}_0$ is given by
\begin{equation} \label{homo-op}
\mathcal{L}_0 = -\text{\rm div} ( \widehat{A} \nabla ),
\end{equation}
where, for $\xi \in \R^d$, 
\begin{equation}\label{homo-co}
\langle \widehat{A}\xi, \xi \rangle = \fint_Y \langle A \nabla v_\xi, \nabla v_\xi \rangle
\end{equation}
and $v_\xi (y)  =v_\xi^A(y)= \langle \xi, y+ \chi (y) \rangle$.
By \eqref{cell-1}  we see that $\mathcal{L}_1 (\chi_j +y_j)=0$ in $\R^d$.
It follows by De Giorgi - Nash estimates that $\chi_j$ is H\"older continuous.
Furthermore,   if $A$ is H\"older continuous, i.e.,  there exist $\alpha \in (0, 1]$ and $M_\alpha>0$ such that 
\begin{equation}\label{H}
|A(x)-A(y) |\le M_\alpha |x-y|^\alpha \quad \text{ for any } x, y \in \R^d,
\end{equation}
 so is $\nabla \chi_j$.

\begin{lemma}\label{a2-thm}
Suppose $A=A(y)$ satisfies \eqref{ellipticity},  \eqref{periodicity} and \eqref{H}.
Let $u_\e\in H^1(B(0, r))$ be a (weak)  solution of $\mathcal{L}_\e (u_\e)=F$ in $B(0, r)$,
where $0<\e,  r< \infty$.
Then
\begin{equation}\label{a2-0}
\| \nabla u_\e \|_{L^\infty(B(0, r/2))} \le C_p \left\{\frac{1}{r} \left(\fint_{B(0, r)} u_\e^2 \right)^{1/2}
+  r \left(\fint_{B(0, r)} |F|^p \right)^{1/2} \right\},
\end{equation}
where $p>d$.
\end{lemma}

\begin{proof}

The interior Lipschitz estimate \eqref{a2-0} 
was proved in \cite{AL-1987}  for second-order elliptic systems in divergence form with periodic coefficients.
The constant $C_p$ depends at most on $d$, $\lambda$, $\Gamma$, $\alpha$, $M_\alpha$, and $p$.
\end{proof}

Observe that if $\mathcal{L}_\e (u_\e)=0$ and $v (x)=u_\e (rx)$ for some $r>0$, then $\mathcal{L}_{\e/r} (v)=0$.
This rescaling property is extremely important to us and  will be used repeatedly in this paper.

\begin{lemma}\label{a4-prop}
Assume that $A$ satisfies the same conditions as in Lemma \ref{a2-thm}.
Let $u_\e\in H^1(B(0, r))$ be a solution of  $\mathcal{L}_\e (u_\e)=0$ in $B(0, r)$, where
$0<\e,  r< \infty$. Then
\begin{equation}\label{a4-0}
\fint_{B(0, r)} u_\e^2 \le C \fint_{\partial B(0, r)} u_\e^2,
\end{equation}
\begin{equation}\label{a4-1}
\max_{B(0, 7r/8)} ( u_\e^2 + r^2 |\nabla u_\e|^2)
\le C  \fint_{\partial B(0, r)} u_\e^2.
\end{equation}
\end{lemma}

\begin{proof}
By rescaling we may assume $r=1$.
For $\omega \in \Sd= \partial B(0, 1)$, define
$$
v_\e (\omega) =\sup\big\{ |u_\e (t \omega)|: \ 0< t< 1 \big\}.
$$
By the nontanegntial-maximal-function estimates for $\mathcal{L}_\e$ in the domain $B(0, 1)$  \cite {AL-1987-ho, KS-1}, we obtain 
$$
\int_{\partial B(0, 1)} v_\e ^2 \le C \int_{\partial B(0, 1)} u_\e ^2,
$$
which yields \eqref{a4-0} for $r=1$.
To see \eqref{a4-1}, we note that
$$
\aligned
\max_{B(0, 7/8)}  ( u_\e^2 + |\nabla u_\e|^2)
 & \le C \| u_\e \|^2_{L^\infty( B(0, 15/16))}\\
& \le C \fint_{B(0, 1)} u_\e^2,
\endaligned
$$
where we have used the Lipschitz estimate \eqref{a2-0}.
This, together with \eqref{a4-0}, gives \eqref{a4-1} for $r=1$.
\end{proof}

\begin{remark} \label{symm-re}

{\rm
Let $S$ be a $d\times d$  invertible constant matrix.
Note that if $\text{\rm div} (A(x/\e)\nabla u_\e)=0$ and $v_\e (x)= u_\e (S^{-1} x)$, then $\text{\rm div} (B(x/\e) \nabla v_\e)=0$, where
$$
B(y)= SA(S^{-1} y) S^T.
$$
The matrix $B$ is elliptic, periodic with respect to the lattice $S (\Gamma)$, and satisfies the same smoothness conditions as $A$.
It is not hard to verify that $v_\xi^B (y)=v^A_{S^T \xi} (S^{-1} y)$.
Consequently, by \eqref{homo-co}, 
$$
\widehat{B}= S \widehat{A} S^T.
$$
We now choose $S$ so that
$$
S^{-1}=(1/2)^{1/2} (\widehat{A} +(\widehat{A})^T)^{1/2}.
$$
Then $S(\widehat{A} +(\widehat{A})^T)S^T=2I$, i.e., $  \widehat{B} +(\widehat{B})^T=2 I$.
Observe that if div$(\widehat{B} \nabla v)=0$, then div$(\widehat{B} + (\widehat{B})^T )\nabla v)=0$,
as $\widehat{B}$ is a constant matrix.
As a result, $v$ is harmonic.
From now on, without loss of generality, we shall assume that $ \widehat{A}+(\widehat{A})^T=2I$ and
solutions of $\mathcal{L}_0 (u_0)=0$ are harmonic.
It is known that if $A$ is symmetric, so is $\widehat{A}$.
In this case, one may assume $\mathcal{L}_0=-\Delta$.
However, the symmetry condition  is not needed in this paper.
}
\end{remark}

\begin{lemma}\label{a1-lemma}
Assume that $A$ satisfies the same conditions as in Lemma \ref{a2-thm}.
Let $u_\e\in H^1(B(0, 1))$ be a solution of $\mathcal{L}_\e(u_\e)=0$ in $B(0, 1)$.
Let $u_0$ be a harmonic function in $B(0, 7/8)$ such that $u_0=u_\e$ on $\partial B(0, 7/8)$.
Then 
\begin{equation}\label{a1-0}
\| u_0\|_{L^\infty (B(0, 7/8))} \le C \| u_\e \|_{L^2(\partial B(0, 1))},
\end{equation}
\begin{equation}\label{a1-0a}
\| u_\e -u_0 \|_{L^2(B(0, 7/8))} \le C \sqrt{\e}  \| u_\e\|_{L^2(\partial B(0, 1))},
\end{equation}
for $0< \e\le 1$.
\end{lemma}

\begin{proof}
Since $\| u_0 \|_{L^\infty(B(0, 7/8))} \le \| u_\e\|_{L^\infty(\partial B(0, 7/8))}$ by the maximum principle,
the inequality \eqref{a1-0} follows from \eqref{a4-1} with $r=1$.
To see \eqref{a1-0a}, we let $\Omega=B(0, 7/8)$ and consider
$$
w_\e =u_\e -u_0 -\e \chi_j (x/\e) \frac{\partial u_0}{\partial x_j} \eta_\e \quad \text{ in } \Omega,
$$
where $\eta_\e$ is a cut-off function in $C_0^1(\Omega)$ such that $\eta_\e=1$ if dist$(x, \partial \Omega)\ge 2c\e$,
$\eta_\e (x)=0$ if dist$(x, \partial \Omega)\le c \e$, and $|\nabla \eta_\e|\le C /\e$.
Then
\begin{equation}\label{H-1-1}
\| w_\e \|_{H^1_0(\Omega)} \le C \sqrt{\e} \| u_\e \|_{H^1(\partial\Omega)}.
\end{equation}
It follows that
$$
\aligned
\| u_\e -u_0 \|_{L^2(\Omega)} & \le 
\| w_\e \|_{L^2(\Omega)}
+ C \e \|\nabla u_0 \|_{L^2(\Omega)}\\
& \le C \sqrt{\e} \| u_\e\|_{H^1(\partial \Omega)}\\
& \le C \sqrt{\e} \| u_\e\|_{L^2(\partial B(0, 1))},
\endaligned
$$
where we have used \eqref{a4-1} for the last inequality.
Finally, we refer the reader to \cite[pp.43-45]{Shen-Book}
for a proof of \eqref{H-1-1}.
\end{proof}

The next theorem provides us with  the approximation of solutions of $\mathcal{L}_\e (u_\e)=0$ by harmonic functions.

\begin{thm}\label{a3-thm}
Assume that $A$ satisfies the same conditions as in Lemma \ref{a2-thm}.
Let $u_\e \in H^1(B(0, r))$ be a solution of $\mathcal{L}_\e (u_\e)=0$ in $B(0, r)$.
There exists  a harmonic function $u_0$  in $B(0, 7r/8)$ such that $u_0(0)=u_\e(0)$ and 
\begin{equation}\label{as-00}
\| u_0 \|_{L^\infty (B(0, 7r/8))} 
\le C \left(\fint_{\partial B(0, r)} u_\e^2 \right)^{1/2},
\end{equation}
\begin{equation}\label{a3-0}
\| u_\e -u_0\|_{L^\infty(B(0, 3r/4)) }
\le C \left(\frac{\e}{r} \right)^{1/2}
\left(\fint_{\partial B(0, r)}  u_\e^2 \right)^{1/2},
\end{equation}
\begin{equation}\label{a3-1}
\aligned
  r \| \nabla u_\e  - (I+\nabla \chi(x/\e))\nabla u_0\|_{L^\infty(B(0, 3r/4))}
\le   C \left(\frac{\e}{r} \right)^{1/2} 
\left(\fint_{\partial B(0, r)}   u_\e^2 \right)^{1/2},
\endaligned
\end{equation}
for $0< \e \le  r< \infty$.
\end{thm}

\begin{proof}

By rescaling we may assume $r=1$.
Let $v_0$ be a harmonic function in $B(0, 7/8)$ such that  $v_0 =u_\e$ on $\partial B(0, 7/8)$.
First, we  show that \eqref{as-00}, \eqref{a3-0} and \eqref{a3-1} hold with $v_0$ in the place of $u_0$.
 
The inequality \eqref{as-00} is given by \eqref{a1-0}.
To prove \eqref{a3-0}, we use the following estimate,
\begin{equation}\label{L-infty}
\aligned
\| u_\e -v_0 \|_{L^\infty (B(0, 3/4))}
 & \le C \| u_\e -v_0 \|_{L^2(B(0, 13/16))}
+ C \e \| \nabla v_0\|_{L^2( B(0, 13/16))}\\
&\qquad\qquad
+ C\e \|\nabla^2 v_0\|_{L^2(B(0, 13/16))},
\endaligned
\end{equation}
which holds as long as $\mathcal{L}_\e (u_\e) =\mathcal{L}_0 (v_0)$ in $B(0, 13/16)$.
The estimate \eqref{a3-0} with $r=1$ follows readily from \eqref{L-infty}, \eqref{a1-0a} and \eqref{as-00} as well as interior estimates for harmonic functions.
We refer the reader to \cite[pp.90-91]{Shen-Book} for a proof of \eqref{L-infty}.

To prove  \eqref{a3-1}, we use the following estimate,
\begin{equation}\label{Lip-1}
\aligned
 & \| \nabla u_\e -  (I+\nabla \chi (x/\e)) \nabla v_0 \|_{L^\infty (B(0, 3/4))}\\
& \qquad \le C  \| u_\e -v_0 \|_{L^2(B(0, 13/16))}
+ C \e \| \nabla v_0\|_{L^\infty (B (0, 13/16))} \\
&\qquad \qquad
+ C \e \ln (2+\e^{-1}) \| \nabla^2 v_0 \|_{L^\infty(B(0, 13/16))}
+ C \e^2 \|\nabla^3 v_0\|_{L^\infty(B(0, 13/16))}.
\endaligned
\end{equation}
We refer the reader to \cite[pp.94-95]{Shen-Book} for its proof.
Note that \eqref{a3-1} with $r=1$ follows from \eqref{Lip-1}, \eqref{a1-0a},  \eqref{as-00}, and the interior estimates for harmonic functions.

Finally,  let $u_0 (x)=v_0(x)- v_0(0) + u_\e (0)$.
Cleary, $u_0$ is harmonic in $B(0, 7/8)$ and $u_0(0) =u_\e (0)$.
It is not hard to check that $u_0$ satisfies \eqref{as-00}, \eqref{a3-0} and \eqref{a3-1}.
\end{proof}

\begin{remark}\label{h-app}
{\rm
Throughout the paper we shall assume that  the solution $u_\e$ is not constant.
Under the assumption that $A$ is Lipchitz continuous, 
this  implies that the harmonic approximation $u_0$ constructed in Theorem \ref{a3-thm} is not constant in $B(0, 7r/8)$. 
For otherwise, $u_\e$ would be constant in $\partial B(0, 7r/8)$.
It follows that $u_\e$ is constant in $B(0, 7r/8)$.
By unique continuation we may conclude that $u_\e$ is constant in $B(0, r)$.
We also point out that the powers of $\e/r$  in \eqref{a3-0} and \eqref{a3-1} are not sharp, 
but  are sufficient for our purpose.
In fact, any positive power would work for our proof.
}
\end{remark}

Fix $\lambda>0$ and a lattice $\Gamma$, let
\begin{equation}
\mathcal{M} (\lambda, \Gamma)
=\Big\{ A=A(y): \ A\text{ satisfies \eqref{ellipticity} and \eqref{periodicity} and } \widehat{A}+(\widehat{A})^T= 2I  \Big\}.
\end{equation}

\begin{thm}\label{c-thm}
Let $u_{\e_j}$ be a solution of $\text{\rm div}( A^j (x/\e_j) \nabla u_{\e_j}) =0$ in $B(0, r_0)$,
where $\e_j \to 0$ and $A^j \in \mathcal{M}(\lambda, \Gamma)$.
Suppose that $\{ u_{\e_j} \}$ is bounded in $L^2(B(0, r_0))$.
Then there exists a subsequence, still denoted by $\{u_{\e_j}\}$, and a harmonic function $u_0$ in $B(0, r_0)$,
such that 
$u_{\e_j} \to u_0$ weakly  in $L^2(B(0, r_0))$ and weakly in $H^1(B(0, r))$  for any $0<r< r_0$.
\end{thm}

\begin{proof}
It is more or less well known that there exists a subsequence, still denoted by $\{ u_{\e_j} \}$, and a function $u_0 \in L^2(B(0, r_0))$, 
such that $u_{\e_j} \to u_0$ weakly in $L^2(B(0, r_0))$,
$\nabla u_{\e_j} \to \nabla u_0 $ and $A^j (x/\e_j) \nabla u_{\e_j} \to \overline{A} \nabla u_0$ weakly in 
$B(0, r)$ for any $0< r< r_0$. See e.g. \cite[pp.22-24]{Shen-Book} for a proof. 
The constant matrix $\overline{A}$ is a limit of $\widehat{A^j}$.
Since $A^j \in \mathcal{M}(\lambda, \Gamma)$, we have $\widehat{A^j}+(\widehat{A^j})^T=2I$.
It follows that $\overline{A} +( \overline{A})^T=2I$ and thus $u_0$ is harmonic in $B(0, r_0)$.
\end{proof}

\begin{remark}\label{re-com}
It follows from the proof of Theorem \ref{a3-thm} that if $\mathcal{L}_\e (u_\e)= \mathcal{L}_0 (u_0)=0$ in 
$B(0, r_0)$,  then
\begin{equation}\label{co-1}
\| u_\e -u_0 \|_{L^\infty(B(0, r))} \le C_r \| u_\e-u_0 \|_{L^2(B(0, r_0))} + O (\sqrt{\e}),
\end{equation}
\begin{equation}\label{co-2}
\| \nabla u_\e -(I +\nabla \chi(x/\e)) \nabla u_0 \|_{L^\infty(B(0, r))}
\le C_r \| u_\e -u_0 \|_{L^2(B(0, r_0))} + O (\sqrt{\e}), 
\end{equation}
 for any $0< r< r_0$.
This shows that under the additional condition that $A^j$ is H\"older continuous, the convergence of
the subsequence $\{ u_{\e_j}\}$ in Theorem \ref{c-thm} can be strengthened to 
\begin{equation}\label{co-3}
\| u_{\e_j}  -u_0 \|_{L^\infty(B(0, r))} \to 0,
\end{equation}
\begin{equation}\label{co-4}
\| \nabla u_{\e_j}  -(I +\nabla  \chi^j (x/\e_j )) \nabla u_0 \|_{L^\infty(B(0, r))}
\to 0, 
\end{equation}
 for any $0< r< r_0$.
where  $\chi^j$ denote  the first-order correctors for the matrix $A^j$. 
\end{remark}

We end this section with an observation regarding the invertibility condition \eqref{inv-0}.
Let $v^\xi (y) =\langle \xi,  y+ \chi (y)\rangle $, where $\xi \in \Sd$.
Suppose that $\text{det}(I+\nabla\chi (y_0))=0$ for some $y_0 \in Y$.
Then $ \nabla v^\xi (y_0)=0$ for some $\xi \in \Sd$.
Let 
$$
u_\e (x)= \e v^\xi (x/\e)=\langle  \xi, x +\e \chi(x/\e)\rangle,
$$
 which satisfies the doubling condition \eqref{doubling-0} if $\e$ is small.
Note that  $\mathcal{L}_\e (u_\e)=0$ in $\R^d$ and 
$$
\nabla u_\e (x) = \nabla v^\xi (x/\e)
=\langle \xi,  I +\nabla \chi (x/\e)\rangle.
$$
Since $\nabla u_\e$ is periodic, 
$$
\big\{ x\in \R^d: \ x=\e (y_0 + z) \text{ for some } z\in \Gamma \big\}\subset \mathcal{C} (u_\e).
$$
It follows that
$$
| \big\{x: \ \text{dist} (x, \mathcal{C} (u_\e)\cap B(0, 1/2))< r \big\}|
\ge c_d \e^{-d} r^d,
$$
if $0< r\le \e$.
As a consequence, the estimate \eqref{H-M-00} cannot hold uniformly for $\e^{d/(d-2)} <  r\le \e$.
This shows that the condition \eqref{inv-0} in Theorem \ref{main-thm}  is necessary for \eqref{H-M-00} with $r> \e^{d/(d-2)}$.


\section{Doubling Conditions}\label{section-3}

We introduce a doubling index for a continuous function $u$ on a ball $B(x_0, r)$, defined by 
\begin{equation}\label{d-0}
\wN (u, x_0, r)
=\log_ 4 \frac{\fint_{\partial B(x_0, r)} (u-u(x_0))^2 }{\fint_{\partial B(x_0, r/2)} (u-u(x_0))^2},
\end{equation}
assuming $\| u-u(x_0)\|_{L^2(\partial B(x_0, t))} \neq 0$ for $0< t\le r$.
If  $u$ is a (non-constant)  harmonic function  in $B(x_0, r)$, then
\begin{equation}\label{d-01}
\wN(u, x_0, r)=\frac{1}{\ln 2} \int_{r/2}^r \frac{N(u, x_0, s)}{s} \, ds,
\end{equation}
where $N(u, x_0, r)$ denotes Almgren's  frequency  for $u-u(x_0)$ on $B(x_0, r)$, given by 
\begin{equation}\label{fq}
N(u, x_0, r) =\frac{ r\int_{B(x_0, r)} |\nabla u|^2}{\int_{\partial B(x_0,r)} (u-u(x_0))^2}.
\end{equation}
Since $N(u, x_0, r)$ is nondecreasing in $r$, it follows that 
\begin{equation}\label{d-02}
N(u, x_0, r/2) \le \wN (u, x_0, r) \le N(u, x_0, r)
\end{equation}
if $u$ is harmonic in $B(x_0, r)$.

The doubling index for $u_\e$, in general, may not be monotone in $r$.
The next theorem provides a substitute for the monotonicity, which plays a key role in the study of geometric properties of
solutions of elliptic equations.
We point out that a similar result was proved by the present authors in \cite{LS-Nodal}, using averages on balls instead of spheres.
The result on the doubling condition in \cite{LS-Nodal} was the first one for elliptic operators with periodic coefficients.
Subsequent work on  doubling inequalities and the closely related  three-ball lemma, which focus on the explicit 
dependence on $\ell$ of doubling constants,  may be found in \cite{Kenig-Zhu-2021, Kenig-ZZ-2021, Armstrong-2021}.

Define
\begin{equation} \label{M-2}
\mathcal{M}(\lambda, \Gamma, M)
=\Big\{ A=A(y): \  A \text{ satisfies \eqref{ellipticity}, \eqref{periodicity}, \eqref{smoothness}, and } \widehat{A}+(\widehat{A})^T=2I \Big\}.
\end{equation}

\begin{thm}\label{d1-thm}
Let  $L\ge 2$ and $\delta_0\in (0, 1/2]$.
Assume that $A\in \mathcal{M}(\lambda, \Gamma, M)$.
There exists $\e_0=\e_0 (L, \delta_0)>0$ such that  if $0< \e< \e_0r $ and $u_\e \in H^1(B(x_0, r))$ is a non-constant solution of
$\mathcal{L}_\e (u_\e)=0$ in $B(x_0, r)$ for some $r>0$ and $x_0 \in \R^d$, with the properties that,
\begin{equation}\label{d-1}
\wN(u_\e, x_0, r)\le  L+1 \quad \text{ and } \quad
\wN (u_\e, x_0, r/2) \le \ell \pm \delta_0,
\end{equation}
where $\ell \in \mathbb{N}$ and $1\le \ell \le L$,
 then 
 \begin{equation}\label{d-2}
 \wN(u_\e, x_0, s)\le \ell \pm \delta_0 \quad \text{ for } \ \  \frac{r}{8}  \le s \le \frac{r}{4}.
 \end{equation}
 Furthermore, if $0< \e< \e_0 r/2$, then 
 \begin{equation}\label{d-3}
 N^* (u_\e, x_0, s)\le \ell \pm \delta_0 \quad \text{ for } \ \ \frac{\e}{ 2 \e_0} \le s \le \frac{r}{4}.
 \end{equation}
 \end{thm}

\begin{proof}
By translation and dilation we may assume $x_0=0$ and $r=2$. We may also assume $u_\e (0)=0$.
We will show that if $0< \e< \e_0 (L, \delta_0)$ and $u_\e$ is a non-constant solution of $\mathcal{L}_\e (u_\e)=0$ in 
$B(0, 2)$ with the properties that  $u_\e (0)=0$, $N^*(u_\e, 0, 2) \le L+1$ and $N^* (u_\e, 0, 1)\le \ell \pm \delta_0$,
where $\ell \in \mathbb{N}$ and $1\le \ell \le L$, then 
\begin{equation}\label{d-30}
N^*(u_\e, 0, s)\le \ell \pm \delta_0 \quad \text{  for  }\ \  (1/4)\le s\le (1/2).
\end{equation}
By rescaling it follows that if $0< \e< 2^{-j} \e_0$ for some $j \ge 0$, $N^*(u_\e, 0, 2^{-j+1})\le L+1$ and $N^* (u_\e, 0, 2^{-j})\le \ell \pm \delta_0$,
then
$$
N^*(u_\e, 0, s) \le \ell \pm \delta_0 \quad \text{ for } \ \ 2^{-j-2} \le s \le  2^{-j-1}.
$$
 By induction,  this implies that if $0< \e < 2^{-J} \e_0$ for some $J\ge  0$, $N^*(u_\e, 0, 2) \le L+1$ and
 $N^* (u_\e, 0, 1) \le \ell \pm \delta_0$, then
 $$
 N^*(u_\e, 0, s) \le \ell \pm \delta_0 \quad \text{ for } \ \  2^{-J-2} \le s \le (1/2).
 $$
 By choosing $J$ such that $2^{-J-1} \e_0 \le  \e < 2^{-J} \e_0$, we obtain \eqref{d-3} with $r=2$ and $x_0=0$.
 
 To show \eqref{d-30} we argue by contradiction.
Suppose there exist sequences $ \{ \e_j \}$, $\{s_j\}$,  $\{A^j\}\subset \mathcal{M}(\lambda, \Gamma, M)$,  and $\{u_{\e_j} \}\subset H^1(B(0, 2))$ such that
$\e_j \to 0$,  $\{s_j \} \subset [1/4, 1/2]$, div$(A^j (x/ \e_j) \nabla u_{\e_j} )=0$ in $B(0, 2)$, 
\begin{equation}\label{d-4}
\wN (u_{\e_j}, 0, 2) \le L+1 \quad \text{ and } \quad \wN (u_{\e_j} , 0, 1)\le  \ell \pm \delta_0,
\end{equation}
but $\wN (u_{\e_j}, 0, s_j  ) > \ell \pm \delta_0$, where $\ell \in \mathbb{N}$ and $1\le \ell \le L$.
We may assume that  $s_j \to s_0\in [1/4, 1/2]$,  and
\begin{equation}\label{d31}
\fint_{\partial B(0, 1/2)} u_{\e_j}^2 =1.
\end{equation}
Using  \eqref{d-4}, \eqref{d31}  and \eqref{a4-0},
we deduce that $\{ u_{\e_j} \}$ is bounded in $L^2(B(0, 2))$.
It follows by Theorem \ref{c-thm}  that there exists a subsequence, which we still denote by $\{u_{\e_j} \}$, and a harmonic function $u_0$ in $B(0, 2)$,
such that $u_{\e_j}  \to u_0$ weakly in $H^1(B(0, 3/2))$.
By \eqref{co-1}, we obtain $u_{\e_j} \to u_0$ uniformly on $B(0, 1)$.
As a consequence,  $u_0(0)=0$,
$$
\fint_{\partial B(0, 1/2)} u_0^2=1 \quad \text{ and } \quad
\wN(u_0, 0, s_0) \ge \ell \pm \delta_0.
$$
In particular,  $u_0$ is not constant in $B(0, 2)$.
The same argument also gives
$$
\wN (u_0, 0, 1) \le \ell \pm \delta_0.
$$
Since $N^*(u_0, 0, r)$ is non-decreasing for harmonic functions, we see that 
$$
\wN(u_0, 0, s_0) =\wN (u_0,0,  1)=\ell \pm \delta_0.
$$
However, this implies that 
 $u_0$ is a homogeneous harmonic polynomial of degree $k$ and that $\wN(u_0, 0, r)=k$ for any $r>0$.
We obtain a contradiction,  as $\ell \pm \delta_0$ is not an integer.
\end{proof}

 \begin{lemma}\label{lemma-d1}
 Fix $L\ge 2$. 
 Let $u, v$ be two continuous functions in $ \{ x\in \R^d:\  |x|\le r \}$
 such that $\| u-u(0) \|_{L^2(\partial B(0, r/2))}, \| v-v(0) \|_{L^2(\partial B(0, r/2))} \neq 0$,  and
 $$
\| u -v \|_{L^\infty (B(0, r) ) }\le  \e\,  \left(\fint_{\partial B(0, r)} (u-u(0))^2 \right)^{1/2}.
$$
 Suppose that
 $
 \wN (u, 0, r)\le L.
 $
Then, if $0< \e< \e_0 (L)$, 
\begin{equation}\label{d-10}
|\wN  (u, 0, r) -\wN (v, 0, r) | \le C \e,
\end{equation}
where $C$ depends on $L$. 
 \end{lemma}
 
 \begin{proof}
 By dilation we may assume $r=1$.
We may also assume
$$
\fint_{\partial B(0, 1/2)}(u-u(0))^2 =1.
$$
Since $\wN(u, 0, 1) \le L$, it follows that 
$$
\| u - v \|_{L^\infty(B(0, 1))}  \le  C \e.
$$
 Note that 
 $$
 \aligned
 \left(\fint_{\partial B(0, 1)} (v-v(0))^2\right)^{1/2}
  & \le  \left(\fint_{\partial B(0, 1)}  (u-u(0))^2\right)^{1/2} +C \e
  \\ &  = 2^{\wN (u, 0, 1)} 
    \left(\fint_{\partial B(0, 1/2) {}} (u-u(0))^2\right)^{1/2} +C \e\\
  &\le \left( 2^{\wN (u, 0, 1 )} + \frac{C \e}{1-C\e} \right)
    \left(\fint_{\partial B(0, 1/2) {}} (v-v(0))^2\right)^{1/2},
\endaligned
 $$
using
 $$
 1 \le  \left(\fint_{\partial B(0, 1/2)}  (v-v(0))^2\right)^{1/2} + C_0\e.
 $$
 Thus, if $0< \e< \e_0(L)$, where $\e_0(L)\in (0, 1/4)$ is chosen so that $C_0\e_0(L) \le  1/2$, then
$$
2^{\wN (v, 0, 1)}
\le 2^{\wN (u, 0, 1)} + C \e.
$$
Hence, 
\begin{equation}\label{d-11}
\wN (v, 0, 1) \le \wN(u, 0, 1) + C_1 \e,
\end{equation}
where $C_1$ depends on $L$.
Note that $\wN(v, 0, 1) \le L +C_1\e$ and
$$
\| u-v\|_{L^\infty(B(0, 1))} \le C \e \, \| v -v(0)\|_{L^2(\partial B(0, 1))}.
$$
The same argument also gives
$$
\wN (u, 0, 1) \le \wN(v, 0, 1) + C_2 \e,
$$
which, together with \eqref{d-11}, yields \eqref{d-10}.
\end{proof}

\begin{lemma}\label{h-drop-lemma}
Let $u\in H^1(B(0, r))$ be  harmonic  in $B(0, r)$.
Suppose that
\begin{equation}\label{h-d-1}
N(u,0, r) \le \ell -\delta
\end{equation}
for some $\ell \in \mathbb{N}$ and $\delta \in (0, 1/2]$.
Then
\begin{equation}\label{h-d-2}
N(u, 0, tr) \le \ell -1 +\delta
\end{equation}
for $t \in (0, \delta \ell^{-1}]$.
\end{lemma}

\begin{proof}
By dilation we may assume $r=1$. We  also assume $\| u\|=1$, where $\| \cdot \|$ denote the norm in $L^2(\Sd)$.
Write
$$
u(\omega) =\sum_{k \ge 0} a_k \psi_k (\omega)
$$
for $\omega\in \Sd$, where $\psi_k $ are  spherical harmonics of order $k$ with $\| \psi_k \|=1$. Then
$$
N(u, 0, t) =\frac{\sum_k k a_k^2 t^{2k} }{\sum_k a_k^2 t^{2k}}
$$
for $0< t\le 1$. 
Since $\sum_k a_k^2 = \| u\|^2 =1$, we see that 
$$
 \sum_{k \ge \ell } a_k^2 \le\ell^{-1} \sum_{k \ge  \ell } k a_k^2 \le \ell^{-1} N(u, 0, 1) \le 1-\delta\ell^{-1}.
$$
It follows that 
\begin{equation}\label{h-drop-00}
\sum_{k \le \ell-1}  a_k^2 \ge \delta \ell^{-1}.
\end{equation}
Thus, 
$$
\aligned
N(u, 0, t)  & =\frac{\sum_{k \le \ell -1} k a_k^2 t^{2k}}{\sum_k a_k^2 t^{2k}}
+ \frac{\sum_{k \ge \ell } k a_k^2 t^{2k}}{\sum_k a_k^2 t^{2k}}
\\
& \le \ell -1 + \frac{\sum_{k \ge \ell } k a_k^2 t^{2k}}{\sum_{k\le \ell -1}  a_k^2 t^{2k}}\\
&\le \ell -1 +  \frac{ t^{2\ell} \sum_{k \ge \ell } k a_k^2  }{ t^{2(\ell-1)} \sum_{k\le \ell -1}  a_k^2 }\\
&\le \ell -1 + t^2 \ell^2 \delta^{-1}\\
&\le \ell-1 + \delta,
\endaligned
$$
if $t\le \delta \ell^{-1}$, where we have used \eqref{h-drop-00} for the third inequality.
\end{proof}

The next theorem will be used  in an induction argument on $\ell$ in Section \ref{section-7}.

\begin{thm}\label{fq-thm}
Let $L\ge 2$ and $\delta_1 \in (0, 1/2]$.
Assume that $A\in \mathcal{M}(\lambda, \Gamma, M)$.
There exists $\e_1=\e_1 (L, \delta_1)>0$ such that if
$0< \e< \e_1r $,
$u_\e\in H^1(B(x_0, r))$,
$\mathcal{L}_\e (u_\e)=0$ in $B(x_0, r)$ for some $x_0\in \R^d$ and $r>0$,
\begin{equation}\label{d-20}
\wN(u_\e, x_0, r)\le L+1   \quad \text{ and } \quad
\wN(u_\e, x_0, r/2) \le \ell -\delta_1,
\end{equation}
where $\ell \in \mathbb{N}$ and $1\le \ell \le L$, 
then
\begin{equation}
\wN (u_\e, x_0, \delta_1 r/(8\ell) ) \le \ell-1 +\delta_1.
\end{equation}
\end{thm}

\begin{proof}
By translation and dilation it suffices to consider the case $x_0=0$ and $r=1$.
We may also assume $u_\e(0)=0$.
Let $u_0$ be a harmonic function in $B(0, 7/8)$, given by Theorem \ref{a3-thm} with $r=1$.
By Theorem \ref{a3-thm} we have
$$
\aligned
\| u_\e -u_0 \|_{L^\infty(B(0, 1/2))} 
& \le C\sqrt{\e}  \| u_\e \|_{L^2(\partial B(0, 1))}\\
& \le C \sqrt{\e} \| u_\e \|_{L^2(\partial B(0, 1/2))},
\endaligned
$$
where $C$ depends on $L$ and
we have used the condition  $\wN(u_\e, 0, 1) \le L+1$.
By Lemma \ref{lemma-d1}, if $0< \e< \e_1(L)$,
$$
\aligned
\wN(u_0, 0, 1/2)  & \le \wN(u_\e, 0, 1/2)  + C \sqrt{\e} \\
&\le \ell -\delta_1 +C \sqrt{\e} \\
&\le  \ell - {\delta_1/}{2},
\endaligned
$$
where we have assumed $C\sqrt{\e}\le \delta_1/2$.
It follows from \eqref{d-02} that 
$$
N(u_0, 0, 1/4) \le \ell -\delta_1/2.
$$
Since $u_0$ is harmonic, by Lemma \ref{h-drop-lemma},
$$
N(u_0, 0, \delta_1/(8\ell)) \le \ell -1 +\delta_1/2.
$$

Next, let $2^{-J-1} < \delta_1\ell^{-1}  \le 2^{-J}$ for some $J \ge 1$.
Suppose 
$$
0< \e < \delta_1 \e_0/(8\ell)\le \e_0 2^{-J -3},
$$
 where $\e_0>0$ is given by Theorem \ref{d1-thm} with $\delta_0=1/2$.
It follows that 
$$
\wN(u_\e, 0, 2^{-j}) \le \ell +1/2 \quad \text{ for } j=1, 2, \dots,  J+5.
$$
This implies that
$$
\aligned
\left(\fint_{\partial B(0,1/2 )} u_\e^2 \right)^{1/2}
&=   \prod_{j=1}^{J+5} 2^{\wN(u_\e, 0, 2^{-j})}
\left(\fint_{\partial B(0, 2^{-J-6} )} u_\e^2 \right)^{1/2} \\
&\le C 2^{(J+5)(\ell+1)} \left(\fint_{\partial B(0, \delta_1/(8\ell))} u_\e^2 \right)^{1/2}.
\endaligned
$$
Hence,
$$
\aligned
\| u_\e -u_0 \|_{L^\infty(B(0, \delta_1/(8\ell)))}
 & \le C\sqrt{ \e }  \| u_\e \|_{L^2(\partial B(0, 1/2))}\\
& \le C\sqrt{ \e } \left(\fint_{\partial B(0, \delta_1/(8\ell))} u_\e^2 \right)^{1/2},
\endaligned
$$
where $C$ depends on $L$ and $\delta_1$.
As a result, 
we may use Lemma \ref{lemma-d1} to conclude that
$$
\aligned
\wN (u_\e, 0, \delta_1/(8\ell))
 & \le \wN (u_0, 0, \delta_1/(8\ell)) + C\sqrt{ \e} \\
 &\le N(u_0, 0, \delta_1/(8\ell) ) +C \sqrt{\e} \\
& \le \ell -1 + ({\delta_1}/{2} )+ C\sqrt{ \e} \\
&\le \ell -1 +\delta_1,
\endaligned
$$
where we have assumed  $C\sqrt{ \e } \le \delta_1/2$ for the last step.
\end{proof}

In the next theorem we will assume that the periodic matrix  $I+\nabla \chi$ is nonsingular  and satisfies \eqref{inv-0}
for some $\mu>0$.
Define
\begin{equation}\label{A}
\mathcal{A}(\lambda, \Gamma, M, \mu)
=\Big\{ A=A(y): A \text{ satisfies \eqref{ellipticity}, \eqref{periodicity}, \eqref{smoothness},  \eqref{inv-0}   and }
\widehat{A}+\widehat{A}^T =2I  \Big\}.
\end{equation}

\begin{thm}\label{low-L}
Let $L\ge 2$ and  $A\in \mathcal{A} (\lambda, \Gamma, M, \mu)$.
There exists $\e_0=\e_0(L)>0$ such that if 
 $u_\e\in H^1(B(0, 1))$ is  a non-constant solution of $\mathcal{L}_\e(u_\e)=0$ in $B(0, 1)$ for some $0< \e< \e_0$, $\wN(u_\e, 0, 1) \le L$, and
\begin{equation}
\wN(u_\e, 0, 1/2)\le 3/2,
\end{equation}
then $|\nabla u_\e (0)|\neq 0$.
\end{thm}

\begin{proof}
We argue by contradiction.
Suppose there exists a sequence $\{ u_{\e_j} \} \subset H^1(B(0, 1))$ of non-constant  solutions to 
$\text{\rm div}(A^j(x/\e_j) \nabla u_{\e_j})=0$ in $B(0, 1)$ such that  $\e_j \to0 $,  $\wN(u_{\e_j} , 0, 1) \le L$, $\wN(u_{\e_j}, 0, 1/2) \le 3/2$, and
$\nabla u_{\e_j} (0)=0$, 
where $A^j \in \mathcal{A}(\lambda, \Gamma, M, \mu)$.
Without loss of generality we may assume $u_{\e_j} (0)=0$ and $\| u_{\e_j} \|_{L^2(\partial B(0, 1/2))}=1$.
Since $\wN(u_{\e_j} , 0, 1)\le L $, this implies that $\{ u_{\e_j} \}$ is bounded in $L^2(\partial B(0, 1))$ and thus bounded in $L^2(B(0, 1))$, by \eqref{a4-0}.
It follows by Theorem \ref{c-thm} and Remark \ref{re-com}  that there exists a subsequence, still denoted by $\{ u_{\e_j}\}$, and a harmonic function $u_0$ in $B(0, 1)$,
such that  $u_{\e_j}  \to u_0$ weakly in $L^2(B(0, 1))$ and that \eqref{co-3} and \eqref{co-4} hold for any $0< r<1$.

Next, since $u_{\e_j} \to u_0$ uniformly in $B(0, 3/4)$, we see that 
$$
\|u_0\|_{L^2(\partial B(0, 1/2))} =1 \quad \text{ and } \quad
\wN(u_0, 0, 1/2)\le 3/2.
$$
Also observe that $u_0 (0)=\lim u_{\e_j} (0)=0$.
As a result, by the monotonicity of $\wN(u_0, 0, r)$, we may conclude that 
$\lim_{ r \to 0} \wN(u_0, 0, r)=1$ and  that $|\nabla u_0 (0)|\neq 0$.
However, by \eqref{co-4}, 
$$
\nabla u_{\e_j} (0) - (I+\nabla \chi^j (0)) \nabla u_0(0) \to 0,
$$
as $\e_j \to 0$, where $\chi^j $ is the first-order corrector for the matrix $A^j$. 
We obtain a contradiction since $\nabla u_{\e_j} (0)=0$ and $| (I+\nabla \chi^j)^{-1}(0) | \le C$.
\end{proof}


\section{Estimates of turning, part I}\label{section-4}

Throughout this section we assume $A\in \mathcal{A}(\lambda, \Gamma, M, \mu)$.
For a function $f=f(\omega)$ in $L^2(\Sd)$, we use $P_k(f)$ to denote  the homogeneous harmonic polynomial of degree $k$ whose restriction on $\Sd$ is
the projection  of $f$ onto the subspace of  spherical harmonics of order $k$.
That is, $P_k(f) (r \omega) = r^k P_k (f)(\omega)$ for $r>0$ and $\omega \in \Sd$, and $P_k(f)(\omega)$ is the projection of $f$ onto the subspace of
spherical harmonics of order $k$.

If $f$ is defined in $B(0, r_0)$ and $0< r< r_0$, we will use the notation $f(r, \omega) = f(r\omega)$.
Thus, for a fixed  $r$,
 $P_k (f(r, \cdot ))=P_k (g)$ and $\| f(r, \cdot) \|=\| g\|$, where $g(\omega) = f(r\omega)$ and  $\| \cdot \|$  denotes the norm in $L^2(\Sd)$.

\begin{lemma}\label{h-1-lemma}
Let $u\in H^1(B(0, 1))$ be a non-constant harmonic function  in $B(0, 1)$.
Assume that  $u(0)=0$ and
\begin{equation}\label{h-1-0}
\ell -(1/32)\le N(u, 0, 1/2) \le N(u, 0, 1) \le \ell +(1/32)
\end{equation}
for some $\ell \in \mathbb{N}$. 
Let
$$
\eta=N(u, 0, 1)-N(u, 0, 1/2).
$$
Then 
\begin{equation}\label{h-1-1}
a^2_\ell \ge (1-3 \eta) \| u\|^2,
\end{equation}
\begin{equation}\label{h-1-2}
\sum_{j=0}^{\ell-1} a_j^2 2 ^{-2j} \le 2 \eta 2^{-2\ell} \| u\|^2
\quad \text{ and } \quad
\sum_{j=\ell+1}^\infty a_j^2\le 2 \eta \| u\|^2,
\end{equation}
where $u(\omega)=\sum_{j=0}^\infty a_j \psi_j (\omega)$  for $\omega \in \Sd$,
 $\psi_j (\omega)$ is a spherical harmonic of order $j$ with 
$\| \psi_j \|=1$.
\end{lemma}

\begin{proof}
A similar result was proved in \cite[Lemma 3.18]{Naber-2017}.
We provide a more transparent proof here.
Without loss of generality we may assume   $\| u\|=1$.
By Weiss type monotonicity formulas for harmonic functions (see \cite{Weiss} for $\kappa=2$ and \cite{Garofalo-2009} for the general case),
\begin{equation}\label{weiss}
\frac{d}{dr} \big\{ W_\kappa (u, r)\big\}
=\frac{2}{r^{d+2\kappa}}
\int_{\partial B(0, r)} (x\cdot \nabla u - \kappa u)^2 \ge 0,
\end{equation}
where $\kappa>0$ and 
$$
\aligned
W_\kappa (u, r)
 & =\frac{1}{r^{d-2 +2\kappa} }\int_{B(0, r)} |\nabla u|^2 -\frac{\kappa}{r^{d-1+2\kappa}} \int_{\partial B(0, r)} u^2\\
 &= \left\{ N(u, 0, r) -\kappa \right\} \cdot \frac{1}{r^{d-1+2\kappa}} \int_{\partial B(0, r)} u^2.
 \endaligned
$$
Let $\kappa =N(u, 0, 1/2)$.
By integrating both sides of \eqref{weiss} in $r$  from $1/2$ to $1$, we obtain 
\begin{equation}\label{w-1}
2\int_{B(0, 1)\setminus B(0, 1/2)}
\frac{ ( r \partial_r u - \kappa u)^2}{|x|^{2\kappa}} \frac{dx}{|x|^d}
= \big\{ N(u, 0, 1) - \kappa \big\} \int_{\partial B(0, 1)} u^2.
\end{equation}
Now, write $u(\omega) =\sum_{j=0}^\infty a_j \psi_j (\omega)$, where $\psi_j (\omega)$ is a spherical harmonic of degree $j$ with
$\|\psi_j \|=1$. Then 
$$
\sum_j a_j^2 =\| u\|^2 =1.
$$
Moreover, since $u(x)= u(r\omega)=\sum_j a_j r^j \psi_j (\omega)$,  it follows from \eqref{w-1} by a direct computation  that
\begin{equation}\label{h-1-3}
\sum_j a_j^2 | j -\kappa| | 1- 2^{2(\kappa-j)} |
=N(u, 0, 1) -N(u, 0, 1/2)=\eta.
\end{equation}
Note that if  $j \ge \ell+1$, then $j \ge \kappa + 7/8$ and
$$
|1-2^{2(\kappa-j)}|
\ge 7/8.
$$
By \eqref{h-1-3} we see that
\begin{equation}\label{h-1-4}
 \sum_{j \ge \ell+1} a_j^2 \le  (8/7)^2
\sum_{j\ge \ell+1 } a_j^2  |j-\kappa|  |1-2^{(\kappa-j)} |  \le 2 \eta.
\end{equation}
In the case $j\le \ell -1$, we have $j \le \kappa -7/8$ and 
$$
|2^{2(\kappa-j)}-1|\ge(7/8) 2^{2(\kappa-j)} .
$$
Again, by \eqref{h-1-3}, we obtain 
\begin{equation}\label{h-1-5}
\aligned
\sum_{j\le \ell -1} a_j^2 2^{-2j}
 & \le (8/7)^2  2^{-2\kappa} \sum_{j \le \ell-1} a_j^2 |j-k| | 2^{2(\kappa -j)} -1|\\
&  \le (8/7)^2 2^{-2\ell } \eta 2^{-2(\kappa -\ell)}\\
&  \le 2 \eta 2^{-2 \ell},
\endaligned
\end{equation}
where we have used the fact $|\kappa -\ell|<  (1/32)$.
In particular, this gives
\begin{equation}\label{h-1-6}
\sum_{j \le \ell-1} a_j^2 \le \eta.
\end{equation}
Since $\sum_j a_j^2=1$,  in view of \eqref{h-1-4} and \eqref{h-1-6}, we obtain  $a_\ell^2 \ge 1-3 \eta$.
  The estimates in   \eqref{h-1-2} are contained in \eqref{h-1-4} and \eqref{h-1-5}.
\end{proof}

\begin{remark}\label{remark-p1}
Let $u$ be  the same  as in Lemma \ref{h-1-lemma}.
It follows from \eqref{h-1-1} that 
\begin{equation}\label{r-p-1}
\| u  -P_\ell (u)\| \le  C\sqrt{\eta} \| u\|.
\end{equation}
By  interior estimates for harmonic functions,
$$
|\nabla u  (x) -\nabla  P_\ell (u) |
\le C\sqrt{\eta}  \| u\|.
$$
for any $x\in B(0, 3/4)$.
Hence,
$$
\aligned
|\nabla u(x)|
 & \ge |\nabla  P_\ell (u)| - C \sqrt{\eta}  \| u\|\\
 & \ge c_0 \{ |\nabla_\omega P_\ell (u) (\omega) | + |P_\ell (u) (\omega) | \}  - C_0   \sqrt{\eta} \| u\|,
 \endaligned
$$
 where $|x|=1/2$ and  $\omega =x/|x|$. 
Since $\| u(1/2, \cdot )\|\sim 2^{-\ell } \| u\|$ and $P_\ell (u(1/2, \cdot )) = 2^{-\ell} P_\ell (u)$, we obtain
\begin{equation}\label{r-p-2}
\frac{|\nabla u (1/2, \omega)| }{\| u(1/2, \cdot )\|} 
\ge \frac{c_0 \{ |\nabla_\omega P_\ell (u(1/2, \cdot ))| +| P_\ell (u(1/2, \cdot ))| \}} { \| u(1/2, \cdot ) \|}
- C_0 \sqrt{\eta}
\end{equation}
for any $\omega \in \Sd$, 
where $c_0, C_0>0$ depend on $\ell$.
\end{remark}

\begin{remark}\label{remark-p10}
Let $u$ be the same as in Lemma \ref{h-1-lemma}.
Consider the harmonic function 
$$
v(x)=u(x)-2^\ell u(x/2)=\sum_{k\neq \ell} a_k \psi_k (x) -\sum_{k \neq \ell} a_k 2^{\ell-k} \psi_k (x).
$$
It follows by Lemma \ref{h-1-lemma} that
$$
\| v\|^2 =\sum_{k \neq \ell} a_k^2 (1-2^{\ell-k})^2\le C \eta\| u\|^2.
$$
By interior estimates for harmonic functions,
$$
|\nabla v(x)| \le C \sqrt{\eta} \| u\|
$$ 
for any $x\in B(0, 3/4)$. This implies that
$$
2^{\ell-1} |\nabla u(x/2)| \ge |\nabla u(x)| -C \sqrt{\eta} \| u\|
$$
for any $x\in B(0, 3/4)$.
By induction we obtain 
$$
2^{j (\ell-1)} |\nabla u(\omega/2^{j+1})|
\ge |\nabla u (\omega/2)| - C(j, \ell) \sqrt{\eta} \| u(1/2, \cdot) \|
$$
for any $\omega \in \Sd$ and $j \ge 0$, where $C(j, \ell)$ depends on $j$ and $\ell$.
This, together with \eqref{r-p-2}, gives
\begin{equation}\label{r-p-20}
\frac{2^{j (\ell-1)} |\nabla u(1/2^{j+1}, \omega )|} { \|u(1/2, \cdot)\|}
\ge \frac{c_0 \{ |\nabla_\omega P_\ell (u(1/2, \cdot ))| +| P_\ell (u(1/2, \cdot ))| \}} { \| u(1/2, \cdot ) \|}
- C(j, \ell)\sqrt{\eta}
\end{equation}
for any $\omega \in \Sd$ and $j \ge 0$, 
where $c_0$ depends on $\ell$, and $C(j, \ell) $ on $j$ and $\ell$.
\end{remark}

\begin{lemma}\label{h-2-lemma}
Let $u$ be a harmonic function in $B(0, 1)$ satisfying the same conditions  as in Lemma \ref{h-1-lemma}.
Then
\begin{equation}\label{h-2-0}
\Big \| \frac{P_\ell (u )}{\| u \|} -\frac{P_{\ell} (u( 1/2, \cdot ))}{ \| u (1/2, \cdot )\| }\Big  \|
\le 8 \eta,
\end{equation}
where $\eta= N(u, 0, 1)-N(u, 0, 1/2)$.
\end{lemma}

\begin{proof}
Write
$$
u(\omega) =\sum_k  P_k ( u) =\sum_k a_k \psi_k (\omega),
$$
where $\| \psi_k \|=1$.
Then
$$
u ( \omega/2) =\sum_k a_k 2^{-k}  \psi_k (\omega).
$$
Without loss of generality we may assume 
$$
\| u \|^2 =\sum_k a_k^2 =1.
$$
Note that 
\begin{equation}\label{h-2-1}
\aligned
\Big \| \frac{P_\ell (u )}{\| u\|} -\frac{P_{\ell} (u(1/2, \cdot ))}{ \| u (1/2, \cdot )\| }\Big \|
& =\Big | a_\ell -\frac{a_\ell  2^{-\ell}}{\| u ( 1/2, \cdot )\|}\Big |\\
&  \le \Big |  1 -\frac{2^{-\ell}}{\| u ( 1/2, \cdot )\|}\Big |\\
&=\Big |\frac{ \| u ( 1/2, \cdot )\| -2^{-\ell}}{\| u ( 1/2, \cdot )\|} \Big|\\
&= \frac{ | \| u (1/2, \cdot )\|^2 - 2^{-2\ell}| }{  \| u ( 1/2, \cdot )\|
( \| u( 1/2, \cdot ) \| + 2^{-\ell})}.
\endaligned
\end{equation}
By Lemma \ref{h-1-lemma}, 
$$
a_\ell^2 \ge 1-3 \eta \ge 3/4,
$$
$$
\sum_{k=0}^{\ell-1} a_k^2  2^{-2k}  \le 2  \eta 2^{-2\ell} \quad \text{ and } \quad
\sum_{k=\ell +1}^\infty a_k^2 \le 2 \eta.
$$
It follows that
$$
\aligned
| \| u(1/2, \cdot )\|^2 - 2^{-2\ell} |
&= \sum_{k \le \ell -1} a_k^2 2^{-2k}
+ (1-a_\ell^2) 2^{-2\ell}
+ \sum_{k \ge \ell+1} a_k^2 2^{-2k}\\
&\le 6 \eta 2^{-2\ell},
\endaligned
$$
and that
$$
\| u(1/2, \cdot )\|\ge |a_\ell| 2^{-\ell}\ge 2^{-\ell-1}.
$$
This, together with \eqref{h-2-1}, gives \eqref{h-2-0}.
\end{proof}

The next lemma is elementary.

\begin{lemma}\label{n-lemma}
Let $f, g \in L^2(\Sd)$and $\|f\|, \|g \|\neq 0$.
Suppose
$$
\| f-g \| \le \delta \| f\|
$$
for some $\delta \in (0, 1]$.
Then
$$
\Big\| \frac{f}{\| f\|} -  \frac{g} {\| g \| }\Big \| \le \sqrt{2} \delta.
$$
\end{lemma}

\begin{proof}
Note that 
$$
\| f-g \|^2= \| f\|^2 -2\langle f, g\rangle + \|g\|^2 \le \delta^2 \| f\|^2.
$$
It follows that 
$$
2\langle f, g\rangle  \ge  (1-\delta^2) \|f \|^2 +  \|g\|^2\ge 2 \sqrt{(1-\delta^2} \| f\| \| g\|.
$$
Hence.
$$
\aligned
\Big \| \frac{f}{\| f\|} -  \frac{g} {\| g \| }\Big\|^2 & = 2-\frac{2\langle  f, g\rangle }{\| f\| \|g \|}\\
&\le 2- 2 \sqrt{1-\delta^2}\\
&= \frac{2\delta^2}{1+\sqrt{1-\delta^2}}\\
&\le 2\delta^2.
\endaligned
$$
\end{proof}

We now prove the estimate \eqref{h-2-0} for solutions $u_\e$ by harmonic approximation.

\begin{lemma}\label{lemma-p1}
Fix  $L\ge 2$.
Let $u_\e\in H^1(B(0, 4))$ be a non-constant  solution of $\mathcal{L}_\e (u_\e)=0$ in $B(0, 4)$
such that $u_\e (0)=0$.
Suppose that $\wN(u_\e, 0, 1), \wN(u_\e, 0, 4)\le L+1 $, and
\begin{equation}\label{p1-0}
 \wN(u_\e, 0, 1/2),  \wN (u_\e, 0, 2)   \in [ \ell -(1/64), \ell+ (1/64)],
\end{equation}
 where $\ell \in \mathbb{N}$ and $\ell \le L$.
Then, if $0< \e< \e_2(L)$,
\begin{equation}\label{p1-1}
\aligned
&\Big \| \frac{P_\ell (u_\e )}{\| u_\e \|}
- \frac{P_\ell (u_\e (1/2, \cdot ))}{\| u_\e (1/2, \cdot )\|}\Big \|\\
 & \le 8 \left\{ \wN (u_\e, 0, 2) - \wN (u_\e, 0, 1/2)   \right\} + C\sqrt{\e},
\endaligned
\end{equation}
where $C$ depends on $L$.
\end{lemma}

\begin{proof}

Let $u_0$ be the harmonic function in $B(0, 7r/8)$, given by Theorem \ref{a3-thm} with $r=4$.
Then $u_0(0)=u_\e(0)=0$, and 
$$
\| u_\e  -u_0  \|_{L^\infty( B(0, 2) )} \le C\sqrt{ \e } \| u_\e \|_{L^2(\partial B(0, 4))}.
$$
Since $\wN(u_\e, 0, 1), \wN(u_\e, 0, 2) , \wN(u_\e, 0, 4)\le L+1$,
it follows that
$$
\aligned
\|u_\e - u_0  \|  & \le C\sqrt{\e }  \| u_\e \|,\\
\| u_\e (1/2, \cdot ) -u_0 (1/2, \cdot ) \|  & \le C\sqrt{ \e } \| u_\e (1/2, \cdot )\|,
\endaligned
$$
where $C$ depends on $L$.
By Lemma \ref{n-lemma} we obtain 
$$
\aligned
\Big\|  \frac{u_\e }{\| u_\e \|} - \frac{u_0 }{\| u_0 \|}\Big \|
& \le C\sqrt{ \e },\\
\Big\|  \frac{u_\e (1/2, \cdot )}{\| u_\e (1/2, \cdot )\|} - \frac{u_0 (1/2, \cdot )}{\| u_0 (1/2, \cdot )\|} \Big\|
 & \le C\sqrt{ \e }.
\endaligned
$$
Consequently,  since $P_\ell$ is a projection, 
\begin{equation}\label{p1-2}
\Big\|  \frac{P_\ell (u_\e ) }{\| u_\e  \|} - \frac{P_\ell (u_0) }{\| u_0 \|}\Big \|
\le C\sqrt{ \e },
\end{equation}
\begin{equation}\label{p1-3}
\Big\|  \frac{P_\ell ( u_\e (1/2, \cdot )) }{\| u_\e (1/2, \cdot )\|} - \frac{P_\ell ( u_0 (1/2, \cdot )) }{\| u_0 (1/2, \cdot)\|} \Big\|
\le C \sqrt{ \e } .
\end{equation}

Next, by Lemma \ref{h-2-lemma},
\begin{equation}\label{p1-4}
\Big\|  \frac{P_\ell ( u_0 ) }{\| u_0 \|} - \frac{P_\ell ( u_0 (1/2, \cdot )) }{\| u_0 (1/2, \cdot )\|}\Big \|
\le 8 \eta,
\end{equation}
where
\begin{equation}\label{p1-5}
\eta = N(u_0, 0, 1) -N(u_0, 0, 1/2)\le \wN(u_0, 0, 2) -\wN (u_0, 0, 1/2)
.\end{equation}
Finally, by Lemma \ref{lemma-d1}, if $0< \e< \e_0(L)$,
$$
\wN(u_0,0,  2) \le \wN(u_\e, 0, 2) + C \sqrt{\e } \quad \text{ and }\quad 
\wN (u_\e, 0, 1/2) \le \wN(u_0, 0, 1/2) +C \sqrt{\e}.
$$
It follows that
$$
\eta\le  \wN(u_\e, 0, 2) -\wN (u_\e, 0, 1/2) + C \sqrt{\e}.
$$
This, together with \eqref{p1-2}, \eqref{p1-3}, \eqref{p1-4} and \eqref{p1-5}, gives \eqref{p1-1}, provided that $C\e_2(L) < 1/64$ and
$\e_2(L) \le \e_0(L)$.
\end{proof}

The next theorem contains  one of the most important estimates in  this paper.
It allows us to reach down to a minimal scale $r$ as long as  $r\ge C_0\e$ and that  the drop of the doubling index from $1$ to $r$ is small.
An inspection of its proof shows that the factor  $\eta$ in the right-hand side of \eqref{h-2-0}  is crucial.
The argument fails if $\eta$ is replaced by $\eta^\alpha$ for some $\alpha<1$.

\begin{thm}\label{p-thm}
Fix $L \ge 2$ and
let $\e_2(L)>0$ be given by Lemma \ref{lemma-p1}.
Let $u_\e\in H^1(B(0, 1))$ be a non-constant  solution of $\mathcal{L}_\e (u_\e)=0$ in $B(0, 1)$
such that $u_\e (0)=0$.
Suppose that $\wN(u_\e, 0, 1) \le L+1$ and
\begin{equation}\label{p2-0}
\wN(u_\e, 0, 2^{-j}) \in [ \ell -(1/64), \ell + (1/64)]
\end{equation}
for $j=1, 2, \dots J+3$, where $J\ge 0, \ell \in \mathbb{N}$ and $\ell \le L$.
Then, if $0< \e< 2^{-J-2} \e_2(L)$, 
\begin{equation}\label{p2-1}
\aligned
&\Big \| \frac{P_\ell (u_\e(1/4, \cdot ))}{\| u_\e (1/4, \cdot )\|}
-\frac{P_\ell (u_\e (1/2^{J+3}, \cdot ))} { \| u_\e (1/2^{J+3}, \cdot ) \|}\Big \| \\
& \le 8 \Big\{ \wN(u_\e, 0, 1/2) + \wN (u_\e, 0, 1/4) -\wN (u_\e, 0, 2^{-J-2})
-\wN (u_\e, 0, 2^{-J-3}) \Big\}\\
&\qquad\qquad\qquad
+C(2^J \e)^{1/2},
\endaligned
\end{equation}
where $C$ depends on $L$.
\end{thm}

\begin{proof}
The case $J=0$ follows readily  from  Lemma \ref{lemma-p1} by a rescaling argument.
Indeed, let $v(x)=u_\e (x/4)$.
Then $\mathcal{L}_{4\e} (v)=0$ in $B(0, 4)$.
Note that $ 4\e < \e_2(L)$ and $\wN(v, 0, r) = \wN(u_\e, 0, r/4)$.
By a change of variables,  we obtain the  estimate \eqref{p2-1}  from \eqref{p1-1}.

The general case also uses  a rescaling argument.
Consider $\phi (x) =u_\e (2^{-j} x)$, where $0\le j \le J$.
Note that
$$
\mathcal{L}_{2^j \e} (\phi )=0 \quad \text{ in } B(0, 1) \quad \text{ and } \quad
\wN(\phi , 0, r)=\wN(u_\e, 0, 2^{-j} r).
$$
Since $2^j \e \le 2^J \e < 2^{-2} \e_2 (L)$, and 
$$
\wN (\phi, 0, 2^{-k}) =\wN(u_\e, 0, 2^{-k-j}) \in [\ell - (1/(64), \ell + (1/64)]
$$
for $k=1, 2, 3$,  by the estimate for the case $J=0$, we obtain 
\begin{equation}\label{p2-2}
\aligned
& \Big\| \frac{P_\ell (\phi(1/4, \cdot ))}{\| \phi (1/4, \cdot )\|}
-\frac{P_\ell ( \phi (1/2^{3}, \cdot ))} { \|  \phi(1/2^{3}, \cdot ) \|} \Big\| \\
& \le 8 \left\{ \wN(\phi, 0, 1/2) 
-\wN (\phi , 0, 2^{-3}) \right\}
+C(2^j \e)^{1/2}.
\endaligned
\end{equation}
By a change of variables this leads to 
\begin{equation}\label{p2-3}
\aligned
&\Big \| \frac{P_\ell (u_\e (1/2^{j+2},\cdot  )) }{\| u_\e (1/2^{j+2}, \cdot )\|}
-\frac{P_\ell (u_\e  (1/ 2^{j+3}, \cdot)) } { \|  u_\e (1/2^{j+3}, \cdot ) \|}\Big \| \\
& \le 8 \left\{ \wN(u_\e, 0, 2^{-j-1} ) 
-\wN (u_\e , 0, 2^{-j-3}) \right\}
+C(2^j \e)^{1/2}
\endaligned
\end{equation}
for any $0\le j \le J$.
By  summing \eqref{p2-3} from $0$ to $J$,
 we see that the right-hand side of \eqref{p2-1} is bounded by
$$
8 \sum_{j=0}^J  \left\{ \wN(u_\e, 0, 2^{-j-1} ) 
-\wN (u_\e , 0, 2^{-j-3}) \right\}
+C\sum_{j=0}^J (2^j \e)^{1/2}.
$$
Observe that the first sum above is a telescope sum, while the second is bounded by $C (2^J \e)^{1/2}$.
This completes the proof.
\end{proof}

\begin{lemma}\label{lemma-p4}
Fix $L\ge 2$.
Let $u_\e\in H^1(B(0, 1))$ be a non-constant  solution of $\mathcal{L}_\e (u_\e)=0$ in $B(0, 1)$
such that $u_\e (0)=0$.
Suppose that $\wN(u_\e, 0, 1)\le L+1$ and
\begin{equation}\label{p4-00}
\wN(u_\e, 0, 2^{-j}) \in [\ell - (1/64), \ell + (1/64)]
\end{equation}
for $j=1, 2, 3$, where $\ell \in \mathbb{N}$ and $\ell \le L$.
Then, if $0< \e< \e_3(L)$,
\begin{equation}\label{p4-0}
\aligned
& \frac{ 2^{k (\ell-1)} |\nabla u_\e(\omega/2^{k+3})|}{ \| u_\e (1/2^3, \cdot )\| }\\
&\quad  \ge 
\frac{c_0\{ |\nabla_\omega P_\ell (u_\e(1/8, \cdot )) (\omega) | + |P_\ell (u_\e (1/8, \cdot ))(\omega) |\} }
{\| u_\e (1/8, \cdot )\|}- C_k( \sqrt{\e}+  \sqrt{\delta}) 
\endaligned
\end{equation}
for any $\omega \in \Sd$ and $k\ge 0$,
where $c_0$ depends on $L$, $C_k$ depends on $k$ and $L$,  and
\begin{equation}
\delta=|\wN(u_\e, 0, 1/2) - \wN(u_\e, 0, 1/8)|.
\end{equation}
\end{lemma}

\begin{proof}
Let $u_0$ be a harmonic function in $B(0,7/8)$, given by Theorem \ref{a3-thm} with $r=1$. 
Then
$$
\aligned
|u_\e (x) -u_0(x)|  & \le C \sqrt{\e}  \| u_\e\|,\\
|\nabla u_\e(x) - (I +\nabla \chi (x/\e) )\nabla u_0 (x)|
& \le C\sqrt{ \e}  \| u_\e\|,
\endaligned
$$
for any $x\in B(0, 3/4)$.
It follows that for any $x\in B(0, 3/4)$,
$$
|\nabla u_\e (x)| \ge c |\nabla u_0 (x)| - C\sqrt{ \e } \| u_\e\|,
$$
where we have used the assumption that the matrix $I + \nabla \chi$ is invertible and $| (I +\nabla \chi)^{-1} | \le C$.
By \eqref{r-p-20} and a simple rescaling, we have
$$
2^{k (\ell-1)} |\nabla u_0 (\omega/2^{k+3} )|
\ge c \{ |\nabla_\omega P_\ell (u_0 (1/8, \cdot ))(\omega) | + |P_\ell (u_0 (1/8, \cdot ))(\omega) |\}- C_k \sqrt{\eta} \| u_0 (1/8, \cdot )\|
$$
for any $\omega\in \Sd$ and $k \ge 0$, where $c$ depends o $\ell$, $C_k$ depends on $k$ and $L$, and 
\begin{equation}\label{p4-1}
\aligned
\eta & = N(u_0, 0, 1/4)-N(u_0, 0, 1/8)\\
& \le \wN (u_0, 0, 1/2)-\wN(u_0, 0, 1/8)\\
& \le \wN(u_\e, 0, 1/2) -\wN (u_\e, 0, 1/8) + C\sqrt{ \e} ,
\endaligned
\end{equation}
and we have used Lemma \ref{lemma-d1} for the last inequality.
As a result, we see that
\begin{equation}\label{p4-2}
\aligned
2^{k(\ell-1)} |\nabla u_\e (\omega/2^{k+3} )|
 & \ge c \{ |\nabla_\omega P_\ell (u_0 (1/8, \cdot ))(\omega)| 
+ |P_\ell (u_0 (1/8, \cdot ))(\omega) |\} \\
&\qquad\qquad
- C_k \sqrt{\eta} \| u_\e (1/8, \cdot )\| - C_k \sqrt{\e}  \| u_\e\|
\endaligned
\end{equation}
for any $\omega\in \Sd$ and $k \ge 0$.

Finally,  observe that $P_\ell (u_\e(1/8, \cdot ) -u_0 (1/8, \cdot ))$ is a spherical harmonic on $\Sd$. Thus, 
$$
\aligned
& \| \nabla_\omega P_\ell \big (u_\e(1/8, \cdot ) -u_0 (1/8, \cdot )\big)\|_{L^\infty(\partial B(0, 1))}
+ \| P_\ell (u_\e(1/8, \cdot ) -u_0 (1/8, \cdot ))\|\\
& \le C \| P_\ell \big(u_\e (1/8, \cdot ) -u_0 (1/8, \cdot )\big)\|\\
& \le C \| u_\e(1/8, \cdot )- u_0(1/8, \cdot )\| \\
& \le C\sqrt{ \e } \| u_\e  \|\\
&\le C \sqrt{\e} \| u_\e (1/8, \cdot)\|.
\endaligned
$$
This, together with \eqref{p4-2} and \eqref{p4-1}, yields \eqref{p4-0}.
\end{proof}

\begin{thm}\label{p5-thm}
Fix $L\ge 2$.
Under the same assumptions on $u_\e$ as in Theorem \ref{p-thm}, 
we have
\begin{equation}\label{p5-0}
\aligned
& \frac{ 2^{-J+ k(\ell-1)} |\nabla u_\e (\omega/2^{k+J+3}) |}{\| u_\e (1/2^{J+3}, \cdot )\|}
\ge \frac{ c\{ |\nabla_\omega P_\ell (u_\e(1/4, \cdot )) (\omega) | + |P_\ell (u_\e (1/4, \cdot ))(\omega) |\} }
{\| u_\e (1/4, \cdot )\|}\\
& \qquad-  C \Big\{ \wN(u_\e, 0, 1/2) + \wN (u_\e, 0, 1/4) -\wN (u_\e, 0, 2^{-J-2})
-\wN (u_\e, 0, 2^{-J-3}) \Big\}\\
&\qquad- C_k (\sqrt{2^J \e} +\sqrt{\delta})
\endaligned
\end{equation}
for any $\omega\in  \Sd$ and $k \ge 0$, provided $0< \e< 2^{-J} \e_4(L)$, where $C, c>0$ depend on $L$, 
$C_k$  depends on $k$ and $L$, and $\delta$ is given by \eqref{delta}
\end{thm}

\begin{proof}
Note that the function in the left-hand side of \eqref{p2-1} is a spherical harmonics.
It follows that
$$
\aligned
 & \Big \| \frac{P_\ell (u_\e(1/4, \cdot ))}{\| u_\e (1/4, \cdot )\|}
-\frac{P_\ell (u_\e (1/2^{J+3}, \cdot ) )} { \| u_\e (1/2^{J+3}, \cdot ) \|}\Big \|_{L^\infty(\Sd)}\\
& \qquad
+ \Big \| \frac{\nabla_\omega P_\ell (u_\e(1/4, \cdot ))}{\| u_\e (\omega/4)\|}
-\frac{\nabla_\omega P_\ell (u_\e (1/2^{J+3}, \cdot ))} { \| u_\e (1/2^{J+3}, \dot ) \|}\Big \|_{L^\infty(\Sd)} \\
& \le C \Big \| \frac{P_\ell (u_\e(1/4, \cdot ))}{\| u_\e (1/4, \cdot )\|}
-\frac{P_\ell (u_\e (1/2^{J+3}, \cdot ))} { \| u_\e (1/2^{J+3}, \cdot ) \|}\Big \|. \\
\endaligned
$$
This leads to
\begin{equation}\label{p5-1}
\aligned
 & \frac{|\nabla_\omega P_\ell (u_\e(1/2^{J+3}, \cdot  ))| + |P_\ell (u_\e (1/2^{J+3}, \cdot ))|}
{\| u_\e (1/2^{J+3}, \cdot )\|}\\
&\ge  \frac{c \{ |\nabla_\omega P_\ell (u_\e(1/4, \cdot  ))| + |P_\ell (u_\e (1/4, \cdot ))|\}}
{\| u_\e (1/4, \cdot )\|}
-C \Big \| \frac{P_\ell (u_\e(1/4, \cdot ))}{\| u_\e (1/4, \cdot )\|}
-\frac{P_\ell (u_\e( 1/2^{J+3}, \cdot ))} { \| u_\e (1/2^{J+3}, \cdot ) \|}\Big \|. \\
\endaligned
\end{equation}

Next, let $\phi (x)=u_\e (2^{-J} x)$.
Then 
$$
\mathcal{L}_{2^J \e} (\phi)=0 \quad \text{ in } B(0, 1) \quad \text{ and } \quad
\wN (\phi, 0, r)= \wN (u_\e, 0, 2^{-J} r).
$$
By applying  Lemma \ref{lemma-p4}  to $\phi$, we see that  if $0< \e< 2^{-J} \e_3 (L)$, 
\begin{equation}\label{p5-2}
\aligned
 & \frac{2^{-J+k (\ell-1) }|  \nabla u_\e (\omega/2^{k+ J+3})|}{\| u_\e (1/2^{J+3}, \cdot )\|}\\
& \ge \frac{ c \{ |\nabla_\omega P_\ell (u_\e (1/2^{J+3}, \cdot )) |
+|P_\ell (u_\e (1/2^{J+3},  \cdot )) | \} }
{\| u_\e (1/2^{J+3}, \cdot ) \|}
-C_k (\sqrt{2^J\e} +\sqrt{\delta} ),
\endaligned
\end{equation}
where
\begin{equation}\label{delta}
\delta =| \wN (u_\e, 0, 1/2^{J+1}) -\wN (u_\e, 0, 1/2^{J+3})|.
\end{equation}
The estimate \eqref{p5-0} now follows readily from \eqref{p5-2}, \eqref{p5-1}, and \eqref{p2-1}.
\end{proof}

\begin{remark}
By introducing a factor $2^{-k}$ in the left-hand side of \eqref{p5-0} we will be 
 able to reach down to the scale $2^{-k}  r^*$ for any fixed $k$, where $r^*$ is the minimal scale defined 
by \eqref{rad-1}.
\end{remark}


\section{Estimates of turning, part II}\label{section-5}

Throughout this section we assume $A\in \mathcal{A}(\lambda, \Gamma, M, \mu)$.
We consider the case where the solution $u_\e$ has two critical points whose doubling indices 
are trapped between $\ell-\delta$ and $\ell +\delta$ for some $\ell\ge 2$.

\begin{lemma}\label{lemma-sh}
Let $u$ be a non-constant harmonic function in $B(0, 2)$.
Let $x_1 \in B(0, 1/4)$.
Suppose that 
\begin{equation}\label{sh-0}
\aligned
\ell -(1/32)  \le  &N(u, 0, 1/2) \le N(u, 0, 1) \le \ell + (1/32),\\
\ell - (1/32)\le & N(u, x_1, 1/2)  \le N(u, x_1, 1) \le \ell +(1/32),
\endaligned
\end{equation}
for some  $\ell \in \mathbb{N}$ and $\ell \ge 2$.
Then
\begin{equation}\label{sh-1}
\| P_\ell (u (x_1+\cdot ))
-P_\ell (u)\|
\le C \sqrt{\eta} |x_1| \| P_\ell (u) \|,
\end{equation}
\begin{equation}\label{sh-2}
\| x_1 \cdot \nabla P_\ell (u)\|
\le C\sqrt{\eta} \| P_\ell (u)\|,
\end{equation}
where 
$$
\eta
= N(u, 0, 1)+ N(u, x_1, 1) - N(u, 0, 1/2) -N(u, x_1, 1/2)
$$
and $C$ depends on $\ell$.
\end{lemma}

\begin{proof}  
The proof is similar to that of Lemma 3.22 in \cite{Naber-2017}.
We mention  that analogous estimates have been obtained  earlier in \cite[Section 2] {Lin-1999} for stationary harmonic maps,
using monotonicity formulas.

Without loss of generality we assume $u(0)=0$.
By a rotation we may also  assume $x_1=(t, 0, \dots, 0)$ and $t>0$.
Let
$$
u(x) =\sum_{k=0}^\infty a_k \psi_k (x),
$$
where $\psi_k (\omega)$ are spherical harmonics of order $k$ and $ \| \psi_k \|=1$.
Then
$$
u(x+x_1) =\sum_{k=0}^\infty a_k \psi_k (x+x_1).
$$
Since $x_1=(t, 0, \dots, 0)$, 
$$
\psi_k (x+x_1) = \sum_{i=0}^k \frac{t^i}{i !} (\partial_1)^i \psi_k (x).
$$
It follows that
$$
\aligned
u(x+x_1)
 &=\sum_{k=0}^\infty \sum_{i=0}^k a_k \frac{t^i}{i !} (\partial_1)^i \psi_k(x)  \\
& =\sum_{k=0}^\infty \sum_{i=0}^\infty  a_{k+i} \frac{t^i}{i !} (\partial_1)^i \psi_{k+i} (x).
\endaligned
$$
Thus
$$
\aligned
P_\ell (u(x_1+\cdot ))
& =\sum_{i=0}^\infty a_{\ell+i} \frac{t^i}{i !} (\partial_1)^i \psi _{\ell+i} \\
& =a_\ell \psi_\ell 
+ \sum_{i=1}^\infty a_{\ell+i} \frac{t^i}{i !} (\partial_1)^i \psi_{\ell+i}.
\endaligned
$$
Hence,
$$
\aligned
\|P_\ell( u (x_1 +\cdot )) -P_\ell (u)\|
 & \le |t| \sum_{i=1}^\infty  |a_{\ell+i}| \frac{t^{i-1}}{i !} \| (\partial_1)^i \psi _{\ell +i} \|\\
 &
 \le C | t| \sum_{i=1}^\infty |a_{\ell+i} | \frac{t^{i-1}}{i !} i^i 2^{i/2},
 \endaligned
$$
where we have used  the fact that for a homogenous harmonic polynomial $Q$ of degree $k$,
$$
\|\nabla Q\|^2 = k (2k+d-2) \|Q\|^2.
$$
Since $i! \sim \sqrt{2\pi i} (i/e)^i$, we obtain
$$
\aligned
\|P_\ell( u (x_1 +\cdot )) -P_\ell (u)\|
 & \le C |t| \sum_{i=1}^\infty |a_{\ell +i} | \frac{(e \sqrt{2} |t|)^{i-1}}{\sqrt{i}}\\
 &\le C |t| \sqrt{\eta} \|u\|
 \endaligned
$$
if $|t|<  1/4$, where we have used Lemma \ref{h-1-lemma} for the last inequality.

To see \eqref{sh-2}, note that
$$
\aligned
 P_{\ell-1} (u(x_1 +\cdot ))
& =\sum_{i=0}^\infty a_{\ell-1+i} \frac{t^i}{i!} (\partial_1)^i \psi_{\ell -1 +i}\\
&=a_{\ell-1} \psi_{\ell-1} 
+ a_{\ell} t \partial_1\psi_\ell 
+\sum_{i=2}^\infty a_{\ell-1+i} \frac{t^i}{i!} (\partial_1)^i \psi_{\ell -1 +i}.\\
\endaligned
$$
It follows that
$$
\aligned
|t|  |a_\ell | \| \partial_1 \psi_\ell\|
& \le \| P_{\ell-1} (u(x_1 +\cdot ))\|
+ \| a_{\ell-1}  \psi_{\ell-1} \|
+ \sum_{i=2}^\infty
|a_{\ell-1 +i} | \frac{|t|^i}{i!} \| (\partial_1)^i \psi_{\ell-1 +i} \|\\
& \le C \sqrt{\eta}(  \| u\| + \| u(x_1 +\cdot ) -u(x_1) \| )\\
& \le C \sqrt{\eta} \| u\|,
\endaligned
$$
where we have used the observation 
$$
\| u(x_1 +\cdot ) \| \le C  \| u(x_1+\cdot/2)\| \le C  \| u \|_{L^2(B(0, 1))}\le C \| u\|.
$$
\end{proof}

We now transfer the estimates in Lemma \ref{lemma-sh} to solutions of $\mathcal{L}_\e (u_\e)=0$
by harmonic approximation.

\begin{lemma}\label{lemma-s1}
Fix $L \ge 2$.
Let $u_\e \in H^1(B(0, 2))$ be a solution of $\mathcal{L}_\e (u_\e)=0$ in $B(0, 2)$.
Suppose that
\begin{equation}\label{s-0}
\aligned
 \wN(u_\e, 0, 2^{-j}) & \in [ \ell -\delta, \ell + \delta],\\
 \wN(u_\e, x_1,  2^{-j}) & \in [\ell -\delta, \ell + \delta ]
\endaligned
\end{equation}
for $j=0, 1, 2, 3$ and for some $x_1\in B(0, 1/4) $, where 
$\delta \in (0, 1/64]$,
$\ell\in \mathbb{N}$ and $2\le \ell\le L$.
Then, if $0< \e< \e_5(L)$,
\begin{equation}\label{s-00}
\Big\|
\frac{P_\ell (u_\e (x_1+\cdot/4))}{\| P_\ell (u_\e (x_1 +\cdot /4))\|}
-\frac{P_\ell (u_\e (\cdot /4))}{ \| P_\ell (u_\e (\cdot/4))\|}\Big \|
\le C ( \sqrt{\delta} +\sqrt{\e}),
\end{equation}
and
\begin{equation}\label{s-01}
\| x_1 \cdot \nabla P_\ell (u_\e (\cdot/4))\|
\le C (\sqrt{\delta} +\sqrt{\e}) \| P_\ell (u_\e (\cdot/4))\|,
\end{equation}
where $C$ depends on $L$.
\end{lemma}

\begin{proof}
As before, we  assume $u_\e (0)=0$.
Let $u_0$ be the  harmonic function in $B(0, 7/8)$, given by Theorem \ref{a3-thm} with $r=1$.
Then $u_0(0)=u_\e (0)=0$ and 
$$
| u_\e (x) -u_0 (x)| \le C \sqrt{\e} \| u_\e \|
$$
for any $x\in B(0, 3/4)$. Moreover, by Lemma \ref{lemma-d1},
\begin{equation}\label{s-02}
\aligned
 & N(u_0, 0, 1/2^j)
\le \wN(u_0, 1/2^{j-1})
\le \wN (u_\e, 0, 1/2^{j-1}) + C \sqrt{\e},\\
& N(u_0, 0, 1/2^{j-1})
\ge \wN(u_0, 0, 1/2^{j-1})
\ge \wN(u_\e, 0, 1/2^{j-1}) - C\sqrt{ \e}
\endaligned
\end{equation}
for $j\ge 2$.
It follows from  Lemma \ref{h-1-lemma}  and a simple rescaling that if $\e>0$ is sufficiently small,
$$
\aligned
\| u_\e (\cdot /4) -u_0 (\cdot/4)\|
 & \le C \sqrt{\e }  \| u_\e(\cdot /4)\|\\
 &\le C\sqrt{ \e } \| u_0 (\cdot /4)\|\\
 & \le C\sqrt{  \e }  \| P_\ell (u_0(\cdot /4))\|,
 \endaligned
$$
where $C$ depends on $L$.
This implies that
$$
\| P_\ell (u_\e(\cdot /4)) -P_\ell (u_0(\cdot /4))\|
\le C \sqrt{\e } \| P_\ell (u_0 (\cdot /4))\|.
$$
In view of Lemma \ref{n-lemma}, we obtain
\begin{equation}\label{s-03}
\Big\|\frac{  P_\ell (u_\e(\cdot /4))}{ \| P_\ell (u_\e(\cdot /4)) \| } 
 -\frac{P_\ell (u_0(\cdot /4))}{ \| P_\ell (u_0(\cdot /4))\|} \Big\|
\le C\sqrt{ \e } .
\end{equation}
Observe that by \eqref{a4-0},
$$
\aligned
\int_{\partial B(0, 1)} u_\e^2
 & \le C \int_{\partial B(0, 1/8)}u_\e^2
\le C \int_{B(0, 1/4)} u_\e^2 \\
& \le C \int_{B(x_1, 1)} u_\e^2
\le C \int_{\partial B(x_1, 1)} u_\e^2
\le C \int_{\partial B(x_1, 1/4)} u_\e^2,
\endaligned
$$ 
where $C$ depends on $L$ and we have used the fact $B(0, 1/4)\subset B(x_1, 1)$.
The same argument for \eqref{s-03} also  yields 
\begin{equation}\label{s-04}
\Big\|\frac{  P_\ell (u_\e(x_1+ \cdot/4))}{ \| P_\ell (u_\e(x_1+\cdot /4)) \| } 
 -\frac{P_\ell (u_0(x_1+ \cdot /4))}{ \| P_\ell (u_0(x_1+\cdot/4))\|} \Big\|
\le C\sqrt{ \e }.
\end{equation}
By applying Lemma  \ref{lemma-sh} to $u_0(x/4)$, we obtain 
$$
\Big\|\frac{  P_\ell (u_0(x_1+ \cdot/4))}{ \| P_\ell (u_0(x_1+\cdot4)) \| } 
 -\frac{P_\ell (u_0( \cdot/4))}{ \| P_\ell (u_0(\cdot /4))\|} \Big\|
\le C ( \sqrt{\delta} +\sqrt{ \e} ).
$$
This, together with \eqref{s-03} and \eqref{s-04}, gives \eqref{s-00}.

To see \eqref{s-01},  note  that Lemma \ref{lemma-sh} also gives
\begin{equation}\label{s-05}
\| x_1 \cdot \nabla P_\ell (u_0(\cdot/4)) \| \le C (\sqrt{\delta} +\sqrt{\e}) \| P_\ell (u_0(\cdot/4))\| 
\end{equation}
In view of \eqref{s-03} we have
\begin{equation}\label{s-06}
\Big\|\frac{  \nabla P_\ell (u_\e(\cdot/4))}{ \| P_\ell (u_\e(\cdot/4)) \| } 
 -\frac{\nabla P_\ell (u_0(\cdot/4))}{ \| P_\ell (u_0(\cdot /4))\|} \Big\|
\le C \sqrt{\e} .
\end{equation}
The estimate \eqref{s-01} now follows readily from \eqref{s-05} and \eqref{s-06}.
\end{proof}

\begin{thm}\label{20-thm}
Fix $L\ge 2$.
Let $\delta \in (0, 1/64]$,  $\ell \in \mathbb{N}$ and $2\le \ell \le L$.
Let $u_\e\in H^1(B(0, 2))$ be a non-constant solution of $\mathcal{L}_\e (u_\e)=0$ in $B(0,2)$.
Suppose that
\begin{equation}\label{20-0}
\wN(u_\e, 0, 2^{-j}) \in [ \ell -\delta, \ell +\delta] \quad \text{ for } j=0, 1, \dots, J_0+3,
\end{equation}
\begin{equation}\label{20-1}
\wN(u_\e, x_1, 2^{-j}) \in [ \ell -\delta, \ell +\delta] \quad \text{ for } j=0, 1, \dots,  J_1+3,
\end{equation}
where $x_1 \in B(0, 1/4)$ and $x_1 \neq 0$.
Assume that
\begin{equation}\label{20-2}
|x_1| \ge c_0 (2^{-J_0} + 2^{-J_1} ).
\end{equation}
Then, if $0< \e< \e_6(L)$,
\begin{equation}
\| n  \cdot \nabla P_\ell (u_\e (1/4, \cdot ))\|
\le C \Big\{ \sqrt{\delta} +  \min \big( (2^{J_0} \e)^{1/2},  (2^{J_1} \e)^{1/2}\big)  \Big\}  \| P_\ell (u_\e (1/4, \cdot )) \|,
\end{equation}
where $n= x_1/|x_1|\in \Sd$ and $C$ depends on $L$ and $c_0$.
\end{thm}

\begin{proof}
By Theorem \ref{p-thm} , it suffices to show that 
\begin{equation}\label{20-3}
\Big \|
\frac{ n\cdot \nabla P_\ell ( u_\e (1/2^{k+3}, \cdot))}{ \| u_\e (1/2^{k+3}, \cdot)\|} \Big\|
\le C\Big\{  \sqrt{\delta} +  \min \big( (2^{J_0} \e)^{1/2} ,   (2^{J_1} \e)^{1/2}\big)  \Big\},
\end{equation}
for some $0 \le k \le J_0$.
To this end, we let $t=2^{-k-1}$ such that $0\le k-1 \le \min(J_0, J_1) $ and $ 2^{-k} \sim |x_1|$, 
and consider the function $\phi (x) =u_\e ( tx)$.
Note that $\mathcal{L}_{\e/t} (\phi)=0$ and
$$
\wN (\phi, 0, r) =\wN (u_\e, 0, tr ),
$$
$$
\wN(\phi, x_1/t,  r) =\wN (u_\e, x_1, tr).
$$
It follows by \eqref{20-0} and \eqref{20-1} that 
$$
\wN(\phi, 0, 2^{-j} ), \wN (\phi, x_1/t, 2^{-j}) \in [\ell -\delta, \ell +\delta]
$$
for $j=0, 1, 2, 3$.
This allows us to apply Lemma \ref{lemma-s1} to $\phi$ to obtain
$$
\| (x_1/t) \cdot \nabla P_\ell (\phi (1/4, \cdot)) \|
\le C ( \sqrt{\delta} + (\e/t)^{1/2} ) \| P_\ell (\phi (1/4, \cdot)) \|,
$$
which gives \eqref{20-3}, as  $P_\ell (\phi(1/4, \cdot))= P_\ell (u_\e (1/2^{k+3} , \cdot))$.
\end{proof}


\section{Approximate tangent plane  and Lipschitz properties}\label{section-6}

Throughout this section we assume $A\in \mathcal{A}(\lambda, \Gamma, M, \mu)$.
We begin with two lemmas on homogeneous harmonic polynomials.

\begin{lemma}\label{lemma-sub-1}
Let $\ell\ge 2$ and  $\psi _\ell $ be a homogeneous harmonic polynomial of degree $\ell$ with $\| \psi_\ell\|=1$.
Let 
\begin{equation}\label{sub-0}
I(\eta, \psi_\ell ) =\big\{ n \in \Sd: \ \| n \cdot \nabla \psi _\ell \|\le \eta \big\}.
\end{equation}
Then for any $\gamma >0$, there exists $\eta=\eta(\gamma, \ell )>0$ such that
\begin{equation}\label{sub}
I(\eta, \psi_\ell  ) \subset \big\{ n \in \Sd: \ \text{dist}(n, V)\le  \gamma  \big\}
\end{equation}
for some subspace $V$ of dimension $d-2$ or less.
\end{lemma}

\begin{proof}
See \cite[Proposition 3.24]{Naber-2017}.
\end{proof}

\begin{lemma}\label{lemma-hs4}
For $\ell\ge 2$, there exists $\sigma_0= \sigma_0 (\ell)>0$ such that if $\psi_\ell=\psi_\ell (x)$ is a homogeneous harmonic polynomial of degree $\ell$ such that
$\| \psi_\ell \|=1$ and
$
\| \partial_i \psi_\ell \|\le \sigma
$
for $i=1, 2, \dots, d-2$ and for some $0< \sigma< \sigma_0$, then $\psi_\ell =\phi_\ell +  \varphi_\ell$, where
$\phi_\ell$ and $\varphi_\ell$ are homogeneous harmonic polynomial of degree $\ell$,
$\phi_\ell $ is a function of $x_{d-1}$ and $x_d$, 
$$
\| \phi_\ell \| \ge 1-C \sigma  \quad \text{ and } \quad 
\| \varphi_\ell \|\le C \sigma,
$$
and  $C$ depends on $\ell$.
\end{lemma}

\begin{proof}

See \cite[Lemma 3.27]{Naber-2017}.
\end{proof}

The subspace $V$ in Lemma \ref{lemma-sub-1} is referred to in \cite{Naber-2017} as an almost invariant subspace for $\psi_\ell$.
The next theorem shows that  homogeneous harmonic polynomials $P_\ell (u_\e(z+\cdot/4))$
share a common almost invariant subspace for $z\in E_\e (\ell, \delta)$,  defined in \eqref{sp2-0}.
Theorems  \ref{thm-sub2} and \ref{thm-sub2} as well as estimates in Remark \ref{remark-sp}
have been proved in \cite{Naber-2017}  for the operator $\mathcal{L}_1$  with Lipschitz coefficients.
Similar estimates have been obtained earlier in \cite[Section 2]{Lin-1999} for energy concentration sets of stationary harmonic maps.
The subspaces $V$ are traditionally called the weak or approximate tangent planes in earlier literatures on the
rectifiability and on the singular sets of harmonic maps and minimal surfaces, see \cite{Mattila, Simon-1996, Lin-1999} and references therein. 

\begin{thm}\label{thm-sub2}
Fix $L\ge 2$ and $\gamma\in (0, 1/2)$.
Let $u_\e\in H^1(B(0, 2))$ be a non-constant solution of $\mathcal{L}_\e (u_\e)=0$ in $B(0, 2)$.
Let
\begin{equation}\label{sp2-0}
E_\e (\ell, \delta) =\Big\{ z\in B(0, 1/4):  \wN(u_\e, z, 2^{-j}) \in [ \ell -\delta, \ell+\delta]  \text{ for } j=0,1, 2, 3 \Big\},
\end{equation}
where $\ell \in \mathbb{N}$ and $2\le \ell \le L$.
Then  there exist $\delta_0=\delta_0(L, \gamma)>0$ and $\e_0=\e_0(L, \gamma )$ with the properties that
if $0<\delta<\delta_0$ and $0< \e< \e_0$, there exist $\eta=\eta(\gamma , L)>0$ and a subspace $V$ of dimension $d-2$ or less such that 
\begin{equation}\label{sp2-3}
I (\eta, \psi_{\ell, z} ) \subset \big\{ n \in \Sd: \ \text{\rm dist} (n, V)\le \gamma \big\}
\end{equation}
for any $z\in E_\e (\ell, \delta)$ and $\psi_{\ell, z} =P_\ell (u_\e (z+ \cdot /4))/ \| P_\ell (u_\e (z+ \cdot/4))\|$.
\end{thm}

\begin{proof}
Fix $z_0 \in E_\e (\ell, \delta)$.
By Lemma \ref{lemma-sub-1}, 
there exist $\eta=\eta(\gamma, \ell)>0$ and a subspace $V$ of dimension $d-2$ or less such that
$$
I(2\eta, \psi_{\ell, z_0}) \subset  \big\{ n \in \Sd: \ \text{\rm dist} (n, V)\le \gamma  \big\}.
$$
By Lemma \ref{lemma-s1},  for any $z\in E_\e (\ell, \delta)$, 
$$
\| \psi_{\ell, z} - \psi_{\ell, z_0} \| \le C (\sqrt{\delta} +\sqrt{\e}).
$$
Thus, if $z\in E_\e (\ell, \delta)$ and $n \in I(\eta, \psi_{\ell, z}) $,
$$
\aligned
\| n \cdot \nabla \psi_{\ell, z_0} \|
& \le \| n \cdot \nabla \psi_{\ell, z} \| + \|\nabla (\psi_{\ell, z} -\psi_{\ell, z_0})  \|\\
&\le \eta
+ C \| \psi_{\ell, z} -\psi_{\ell, z_0} \|\\
& \le  \eta + C (\sqrt{\delta} + \sqrt{\e})\\
&\le2 \eta,
\endaligned
$$
provided that  $\delta_0>0$ and $\e_0>0$ are so small that $C(\sqrt{\delta_0} +\sqrt{\e_0}) \le \eta$.
It follows that $n \in I(2\eta, \psi_{\ell, z_0})$ and thus dist$(n, V) \le \gamma $.
\end{proof}

\begin{thm}\label{thm-sub3}
Fix $L\ge 2$ and $\gamma  \in (0, 1/4)$.
Let $u_\e$ be a non-constant solution of $\mathcal{L}_\e (u_\e)=0$ in $B(0, 2)$.
Let $\{B(x_i, r_i) \}$ be a finite collection of disjoint balls with the following proerties,

\begin{enumerate}

\item

$x_i \in B(0, 1/8)$ and $20 r_i \ge  R\, \e$,

\item

$$
\wN(u_\e, x_i, 2^{-j} ) \in [ \ell -\delta, \ell +\delta]
$$
for $j=0, 1, \dots, J_i+3$ and $2^{-J_i} \le 400 r_i$, where $\ell \in \mathbb{N}$ and $2\le \ell \le L$.

\end{enumerate}
Then, if $0< \e< \e_0(L, \gamma)$,  $0< \delta< \delta_0(L, \gamma)$ and $R\ge (\e_0(L, \gamma))^{-1}$,
 there exists a subspace $V$ of dimension $d-2$ or less such that
\begin{equation}\label{Lip-g}
\text{\rm dist} (x_i-x_k, V) \le \gamma |x_i -x_k|
\end{equation}
for any $i, k$.
\end{thm}

\begin{proof}
Note that $\{ x_i\}\subset  E_\e (\ell, \delta)$, where $E_\e (\ell, \delta)$ is defined by \eqref{sp2-0}.
Thus, by Theorem \ref{thm-sub2}, there exist $\eta_0>0$ and  a subspace $V$ of dimension $d-2$ or less such that
\begin{equation}\label{sub3-1}
I(\eta_0, \psi_{\ell, x_i}) \subset \big\{ n \in \Sd: \ \text{\rm dist} (n, V) \le \gamma \big\}
\end{equation}
 for any $x_i$.
Let $x_i\neq x_k$ be two  centers of balls and $n=(x_k-x_i)/ |x_k- x_i|$.
We will show that $n\in I (\eta_0, \psi_{\ell, x_i})$.
By \eqref{sub3-1} this implies  that   dist$(n, V) \le \gamma $, from which \eqref{Lip-g} follows.

To see $n \in I (\eta_0, \psi_{\ell, x_i}) $, 
we use Theorem \eqref{20-thm}. Note that  $|x_i -x_k|< 1/4$.
Since $B(x_i, r_i)\cap B(x_j, r_j) \neq \emptyset$, 
$$
|x_i -x_k | \ge r_i + r_k \ge c_0 (2^{-J_i} + 2^{-J_k}).
$$
It follows by Theorem \ref{20-thm} that 
$$
\aligned
\| n\cdot \nabla \psi_{\ell, x_i} \|
 &\le C \big\{ \sqrt{\delta} + \min \big( (\e/r_i)^{1/2}, (\e/r_k)^{1/2} \big)  \big\} \\
&\le C \big\{ \sqrt{\delta_0} + \sqrt{\e_0} \big\}\\
&\le \eta_0,
\endaligned
$$
provided that $ C (\sqrt{\delta_0} + \sqrt{\e_0}) \le \eta_0$.
\end{proof}

\begin{remark}\label{remark-sp}
Suppose that the conclusion \eqref{Lip-g} in Theorem \ref{thm-sub3} 
 holds for a subspace of dimension $d-2$, but not for any subspace
of dimension $d-3$ or less.
It follows from the proof of Theorem \ref{thm-sub3} that for any $x_i$ and $\eta_0>0$,
\eqref{sub3-1} fails to hold for any subspace of dimension $d-3$ or less.
This implies that
there exists an orthonormal set $\{n_1, n_2, \dots, n_{d-2} \}$ in $I(\eta_0, \psi_{\ell, x_i})$, where  $0< \eta_0< \sigma_0$ and
$\sigma_0=\sigma_0(L)>0$ is given by Lemma \ref{lemma-hs4}.
As a result, 
 we may write $\psi_{\ell, x_i} = \phi_\ell +\varphi_\ell$ such that
$\phi_\ell$ and $\varphi_\ell $ are homogeneous harmonic polynomials,
$\phi_\ell$ is invariant with respect to $V_{d-2}=\text{span} \{n_1, \cdots, n_{d-2} \}$, and that
$$
\|\phi_\ell \|\ge 1-C \eta_0  \quad \text{ and } \quad \| \varphi_\ell \| \le C \eta_0.
$$
It follows that for any $x\in B(0, 2)$,
$$
\aligned
|\nabla \psi_{\ell, x_i}  (x)|
& \ge |\nabla \phi_\ell (x) | - |\nabla \varphi_\ell (x)|\\
& \ge 2\ell  \{ \text{dist}(x, V_{d-2})\}^{\ell-1}  \| \phi_\ell \|
- C \| \varphi_\ell \|\\
&\ge    \{ \text{dist}(x, V_{d-2})\}^{\ell-1}  - C\eta_0.
\endaligned
$$
Also, observe that  by\eqref{s-00} ,
\begin{equation}\label{s-000}
\| \nabla \psi_{\ell, x_k} -\nabla \psi_{\ell, x_i} \|_{L^\infty(B(0, 1))}
\le C ( \sqrt{\delta} +\sqrt{\e}).
\end{equation}
As a result, we obtain
\begin{equation}\label{sp-remark}
|\nabla {\psi}_{\ell, x_k}  (x)|
\ge   \big\{   \text{dist}(x, V_{d-2})\big\}^{\ell-1}  - C\eta_0-C \sqrt{\delta} -C \sqrt{\e} 
\end{equation}
for any $x\in  B(0, 1)$ and for any $x_k$.

By   combining   \eqref{sp-remark} with  Theorem \ref{p5-thm}, we see that
\begin{equation}\label{sp-final}
\aligned
\frac{2^{-j+ m(\ell-1) } | \nabla u_\e (x_k + \omega/2^{j+3+m })| }{ \| u_\e (x_0 + \omega/2^{j+3})\| }
\ge c \{ \text{dist} (\omega, V_{d-2}) \}^{\ell-1}-C \eta_0
-C_m (\sqrt{\delta} + \sqrt{2^J \e} )
\endaligned
\end{equation}
for any $\omega\in \Sd$,  $0\le j \le J_k$ and $m\ge 0$.
Furthermore, by applying the estimate above to  the function $u_\e (tx)$ for $1\le t \le 2$,
we may conclude that if
\begin{equation}\label{sp-f1}
\wN (u_\e, z, r) \in [ \ell -\delta, \ell +\delta] \quad \text{ for } r_0 \le r \le 1,
\end{equation}
where $r_0 \ge C_0 \e$, then if $0< \delta< \delta_0 (L)$ and $0< \e< \e_0 r $,
\begin{equation}\label{sp-f2}
\frac{ r2^{m(\ell-1)} | \nabla u_\e (z +2^{-m} r\omega )| }{ \| u_\e (z + r \cdot  ) \| }
\ge c \{ \text{dist} (\omega, V_{d-2}) \}^{\ell-1}
-C_m (\eta_0+ \sqrt{\delta} + \sqrt{ \e_0} )
\end{equation}
for any $\omega\in \Sd$ and $2r_0\le  r \le 1/4$.
In particular, this implies that  $u_\e$ has no critical points in the set
\begin{equation}\label{sp-f3}
\Big\{ x\in B(x_j, 1/8) \setminus B(x_j , r_j/64): \
\text{dist} ((x-x_j) /|x-x_j|, V_{d-2}) \ge c (\eta_0+ \sqrt{\delta} + \sqrt{\e_0 }) \Big\}.
\end{equation}
Finally, note that since 
$$
V_{d-2} \cap \Sd \subset \big\{ n\in \Sd: \text{dist}(n, V)\le \gamma\big \}, 
$$
we have
$$
\big\{ n \in \Sd: \text{\rm dist}(n, V) \le \gamma \big\}
\subset \big\{n \in \Sd: \text{\rm dist}(n, V_{d-2}) \le c_0 \gamma \big\}
$$
for some $c_0$ depending only on $d$.
As a result, we may  replace the subspace $V$ in Theorem \ref{thm-sub3}  by $V_{d-2}$.
\end{remark}


\section{Proof of Theorem \ref{main-thm}}\label{section-7}

We   fix $L\ge 2$ and let $\e_0=\e_0(L)\in (0, 1/4)$ and $\delta_0=\delta_0(L) \in (0, (1/64))$ be given by Theorem \ref{thm-sub3} with $\gamma=(1/10)$.
We also assume that $\e_0$ is so small that Theorems \ref{d1-thm},  \ref{fq-thm} and \ref{low-L} hold.

\begin{lemma}\label{lemma-cover}
Let $A\in \mathcal{A}(\lambda, \Gamma, M, \mu)$. 
Let  $u_\e\in H^1(B(0, 2))$ be a non-constant solution of $\mathcal{L}_\e (u_\e) =0 $ in $B(0, 2)$.
Suppose that 
\begin{equation}\label{7-0-0}
\wN (u_\e, x, 1) \le \ell + \delta_0 \quad \text{ for any } x\in B(0, 1/2)\cap \mathcal{C} (u_\e),
\end{equation}
where $\ell \in \mathbb{N}$ and $2\le \ell \le L$.
Then,  if  $0< \e< \e_0/2$, 
 there exists a sequence of balls $\{B(x_i, r_i) \}$ such that $x_i \in B(0, 1/2)\cap \mathcal{C} (u_\e)$, $0< r_i\le (1/2)$, and

\begin{enumerate}

\item
\begin{equation}\label{7-0-1}
\mathcal{C}(u_\e)\cap B(0, 1/2)  \subset \bigcup_i B(x_i, r_i/4);
\end{equation}

\item

\begin{equation}\label{sum}
\sum_i r_i^{d-2} \le C;
\end{equation}

\item

For each ball $B(x_i, r_i)$, either $ \e_0^{-1} \e\le   r_i\le 6  \e_0^{-1} \e$, or $ 6 \e_0^{-1} \e  < r_i \le (1/2)$ and 
\begin{equation}\label{7-0-2}
\wN (u_\e, y, c r_i) \le \ell -1 +\delta_0 \quad \text{ for any } y \in B(x_i, r_i/2) \cap \mathcal{C}(u_\e)\cap B(0, 1/2),
\end{equation}
\end{enumerate}
where $C$ and $c$ depend on $L$.
\end{lemma}

\begin{proof}
We adapt a covering argument originated in  \cite{Naber-2017}.
The argument is modified for our homogenization problem and streamlined for clear exposition.

{\bf Step 1. }
For $x\in  B(0, 1/2)\cap \mathcal{C}(u_\e)$, define
\begin{equation}\label{7-0-3}
r^* (x)= \sup \big\{ \e_0^{-1} \e  < s\le 1/2: \wN (u_\e, x, s) \le  \ell -\delta _0 \big\}.
\end{equation}
If no such $s$ exists, i.e., $\wN(u_\e, x, s)>  \ell-\delta_0$ for all $\e_0^{-1} \e < s\le 1/2$, define $r^* (x) =\e_0^{-1} \e$. 
It follows from Theorem \ref{d1-thm}  that if $r^*(x)< 1/2$,
\begin{equation}\label{7-00-3a}
\ell -\delta_0 \le  \wN (u_\e, x, r) \le \ell +\delta_0   \quad \text{ for } \ \  r^*  (x) \le   r\le  1/2, 
\end{equation}
and if $r^*(x)> \e_0^{-1} \e$, 
\begin{equation}\label{7-00-3b}
\wN(u_\e, x, r) \le  \ell -\delta_0 \quad \text{ for } \ \ \e_0^{-1} \e/2 \le  r\le  r^* (x)/2.
\end{equation}
Let
\begin{equation}\label{7-0-5}
\mathcal{C}_g (u_\e)
= \Big\{ x\in \mathcal{C}(u_\e)\cap B(0, 1/2): \ r^* (y) \ge r^*(x)/3 \text{ for any } y \in B(x, r^*(x)) \cap \mathcal{C}(u_\e)\cap B(0, 1/2) \Big\}
\end{equation}
and
$$
\mathcal{C}_b (u_\e) =\mathcal{C}(u_\e) \cap B(0, 1/2) \setminus \mathcal{C}_g (u_\e).
$$
Consider the cover of $\mathcal{C}_g(u_\e)$ by
$$
\Big\{ B(x, r^*(x)/20) : x\in \mathcal{C}_g (u_\e) \Big\}.
$$
Let $\{ B(x_i, r^*(x_i)/4) \}$ be a Vitali subcover, i.e.,
\begin{equation}\label{7-0-6}
\mathcal{C}_g (u_\e) \subset \bigcup_i B(x_i,  r^*(x_i)/4)
\end{equation}
and $B(x_i, r^*(x_i)/20) \cap B(x_j, r^*(x_j )/20 )=\emptyset$ for $i\neq j$.

Let $r_i= r^*(x_i)$. Then either $\e_0^{-1} \e\le r_i\le 6 \e_0^{-1} \e $ or $r_i> 6  \e_0^{-1} \e$.
In the second case, by the definition of $\mathcal{C}_g(u_\e)$,
$$
r^* (y)\ge r_i /3 \quad \text{ for any } y \in B(x_i, r_i)\cap \mathcal{C}(u_\e)\cap B(0, 1/2).
$$
Thus,  $r^* (y) > 2  \e_0^{-1} \e$ for any
 $y\in B(x_i, r_i)\cap \mathcal{C} (u_\e)\cap B(0, 1/2)$. As a result,  since $\e_0^{-1} \e< (1/6) r_i \le  r^* (y)/2$,   in view of \eqref{7-00-3b}, we obtain 
$$
\wN(u_\e, y, r_i/6)  \le \ell -\delta_0.
$$
Hence, in view of  Theorem \ref{fq-thm}, we have proved that either $\e_0^{-1}  \e  \le r_i \le 6\e^{-1}_0 \e $, or $r_i>6  \e^{-1}_0 \e$ and 
$$
\wN (u_\e, y, c r_i) \le \ell -1 +\delta_0 \quad \text{ for any } y \in B(x_i, r_i)\cap \mathcal{C}(u_\e) \cap B(0, 1/2),
$$
where $c= \delta_0/(48\ell)$.

\medskip

{\bf Step 2.} To show \eqref{sum}, 
we  divide the subcover in \eqref{7-0-6} into several groups so that for any $x_i, x_j$ in the same group, we have
$|x_i -x_j| \le 1/(32)$.
Since $B(x_i, r_i/20)\cap B(x_j, r_j /20)=\emptyset$ for $i\neq j$, 
it suffices to consider those $x_i$'s for which $r_i\le (1/32)$.
As a result,  by \eqref{7-00-3a}, we have
$$
\ell -\delta_0\le \wN(u_\e, x_i, t) \le \ell+\delta_0 \quad \text{ for } r_i \le  t \le (1/2).
$$
It follows by Theorem \ref{thm-sub3}   that there exists a subspace $V$ of $\R^d$ of dimension $d-2$ or less such that
$$
\text{dist} (x_i -x_j, V ) \le \frac{1}{10} |x_i -x_j|
$$
for any $x_i, x_j$ in the same group.
This implies that the centers  $\{ x_i \}$ of  balls for each group  lie
 on the  graph of a Lipschitz  function $F: V \to V^\perp$ with Lipschitz  norm less than 
$1/9$.
Consequently, using the fact that $\{ B(x_i, r_i/20) \}$ are disjoint, we obtain \eqref{sum}.

\medskip

{\bf Step 3.}
We construct additional balls with the desired properties  to cover 
$$
\mathcal{C}_b(u_\e)\setminus \cup_i B(x_i, r_i/4).
$$

Let $y\in \mathcal{C}_b (u_\e)$.
There exists $z_0\in B(y,  r^*(y))\cap \mathcal{C}(u_\e)\cap B(0, 1/2)$ such that $ r^* (z_0) < r^*(y)/3$.
If $z_0 \notin \mathcal{C}_g(u_\e)$, then there exists $z_1 \in B(z_0,  r^*(z_0))\cap \mathcal{C} (u_\e)\cap B(0, 1/2)$ such that
$r^*(z_1) < r^*(z_0)/3$. If $z_1 \notin \mathcal{C}_g (u_\e)$, then there exists $z_2 \in B(z_1,  r^*(z_1))\cap \mathcal{C}(u_\e)\cap B(0, 1/2) $, such that
$r^*(z_2) < r^*(z_1)/3$. 
Since $r^*(z) \ge \e_0^{-1}  \e$, 
this process stops in a finite number of steps and yields  some $x\in \mathcal{C}_g(u_\e) $ with properties that
$$
|y-x|\le |y-z_0| + |z_0-z_1|+\cdots 
<  \sum_{k=0}^\infty 3^{-k} r^*(y) = 3 r^*(y)/2,
$$
and $r^*(x)< r^*(y)/3$. Thus $x\in B(x_i,  r_i/4)\cap \mathcal{C}_g(u_\e)$ for some $i$. Since $r^*(x)\ge r_i/3$, we obtain $r_i < r^*(y)$.
It follows that
$$
| y-x_i|
\le |y-x| + |x-x_i|
< 3 r^*(y)/2 + r_i/4 \le 2 r^*(y).
$$
For $y \in \mathcal{C}_b (u_\e)$, define
$$
t(y) =\frac{1}{10} \min_k  | y-x_k | <  r^*(y)/ 5.
$$
Cover $\mathcal{C}_b (u_\e)\setminus \cup_i B(x_i, r_i/4) $ by 
$$
\big\{ B(y, t(y)/20): y \in \mathcal{C}_b (u_\e)\setminus \cup_i B(x_i,  r_i/4)  \big\}
$$
Choose a Vitali subcover 
$$
\{ B(y_j,  t_j/4): j \in J \},
$$
 where $t_j = t(y_j)$, such that $B(y_j, t_j/20) \cap B(y_k, t_k/20) =\emptyset$ for $j \neq k$, and
 $$
 \mathcal{C}_b (u_\e) \setminus \bigcup_i B(x_i, r_i/4)\subset \bigcup_j B(y_j, t_j/4).
 $$
 Thus
 \begin{equation}\label{7-0-10}
 \mathcal{C}(u_\e)\cap B(0, 1/2)
 = \mathcal{C}_g (u_\e) \cup \mathcal{C}_b (u_\e)
 \subset
 \bigcup_i  B(x_i, r_i/4) \cup \bigcup_j B(y_j, t_j/4).
 \end{equation}
 
{\bf Step 4. } We  verify the condition \eqref{7-0-2} for $B(y_j, t_j)$.
Note that $t_j= t(y_j) \ge (1/64) r_i \ge (1/64)\e_0^{-1}  \e$, since $y_j \notin B(x_i, r_i/4)$.
 Suppose $t_j > \e^{-1}_0 \e$. For any $z\in B(y_j, t_j)\cap \mathcal{C}(u_\e) \cap B(0, 1/2)$,
 $$
 |z-x_i|
 \ge |y_j -x_i|- |z-y_j|
 \ge |y_j -x_i| - t_j
 \ge \frac{9}{10} |y_j-x_i|\ge 9 t_j
 $$
 for any $i$.
 Note that if  $z\in B(y_j, t_j)\cap \mathcal{C}_b(u_\e)$,
 $$
 r^*(z) > \frac{1}{2} \min_k |z -x_k |\ge 4 t_j.
 $$
 If $z\in B(y_j, t_j)\cap \mathcal{C}_g(u_\e)$, then $z\in B(x_i, r_i/4)$ for some $i$.
 Since
 $$
 10 t_j \le |y_j -x_i| \le |y_j-z| + |z-x_i| \le t_j + r_i/4,
 $$
 we have $t_j \le r_i /18$ and 
 $$
 r^*(z) \ge r_i /3 \ge 6t_j.
 $$
Hence, in both cases,  $r^*(z) \ge 4 t_j> \e_0^{-1}  \e$.
It follows that 
$$
\wN (u_\e, z,  t_j) \le  \ell -\delta_0,
$$
which leads to
$$
\wN (u_\e, z, c t_j ) \le \ell -1 +\delta_0.
$$ 

{\bf Step 5.} It remains to prove that
\begin{equation}\label{7-0-11}
\sum_{j\in J}  t_j^{d-2} \le C.
\end{equation}
Since $\{B(y_j, t_j/20)\}$ are disjoint, we only need to consider those $t_j$'s that are sufficiently small.
Let $V$ be the subspace of dimension $d-2$ or less, obtained in Step 2.
Suppose the dimension of $V$ is $d-3$ or less.
  Write
 $
 J=\cup_{k=0}^\infty  J_k,
 $
 where $j \in J_k$ if $t_j \in (2^{-k-1}, 2^{-k} ]$.
 Let $\Phi$ be the Lipschitz  graph  that $\{ x_i \}$ lies on.
 Since
 $$
 \text{dist}(y_j, \Phi ) \le 10 t_j,
 $$
 for $j \in J_k$,
  $B(y_j, t_j/20 )$ is contained  in the $11 \cdot 2^{-k}$ neighborhood of $\Phi$.
 Recall that  $\{ B(y_j, t_j/20) \}$ are disjoint.
 It  follows that
 $$
 \# J_k \le C 2^{k (d-3)}.
 $$
 As a result,
 $$
 \sum_j t_j^{d-2} \le \sum_k (2^{-k})^{d-2}  \# J_k \le C.
 $$
 
Finally, we consider the case where dim$(V)=d-2$. It follows by Remark \ref{remark-sp}  that  in this case, the solution $u_\e$ has no critical points in the set
$$
 \big\{ x\in B(x_i, c_0) \setminus B(x_i,  r_i/4): \  \text{dist} (x-x_i , V) > (1/10)  |x-x_i | \big\}.
$$
This implies that if $y\in \mathcal{C}(u_\e)\cap B(x_k, c_0)\setminus \cup_i B(x_i, r_i/4)$ for some $k$, then
$$
\text{\rm dist}(y- x_k, V) \le (1/10)  |y-x_k|.
$$
Since $y_j \in \mathcal{C}(u_\e)\cap B(0, 1/2) \setminus \cup_k B(x_i,  r_i/4)$ and
$\min_i |y_j-x_i|=10 t_j \le c_0$, 
we see that
$$
\text{\rm dist} (y_j-x_k, V) \le (1/10) |y_j -x_k|
$$
for some $k$,  and $|y_j -x_k|\sim t_j$.
As a result,  $\{ \mathbb{P}(B(y_j, t_j/20)): j \in J \}$ has a finite overlap in $V=\R^{d-2}$,
where  $\mathbb{P}: \R^d \to V$ is the projection operator.
Thus,
$$
\sum_j \chi_{\mathbb{P}(B(y_j, t_j/20))} \le C,
$$
and 
$$
\sum_j  \mathcal{H}^{d-2} (\mathbb{P} (B(y_j, t_j/20))) \le C,
$$
from which \eqref{7-0-11} follows.
This completes the proof of Lemma \ref{lemma-cover}
\end{proof}

\begin{lemma}\label{lemma-c-1}
Let $u_\e$ be the same as in Lemma \ref{lemma-cover}.
Then, if $0< \e< \e_0(L)$, 
\begin{equation}\label{c-01}
| \big\{ x: \text{\rm dist} (x, \mathcal{C} (u_\e)\cap B(0, 1/4)) <  t \big\} |
\le C t ^2
\end{equation}
for $0<t< 1$, where $C$ depends on $L$.
\end{lemma}

\begin{proof}

Let $B(x_i, r_i/4)$ be a ball constructed in Lemma \ref{lemma-cover} with radius $r_i> 6\e^{-1}_0 \e$.
We may cover this ball by  a finite number of smaller balls $\{B(z_k, cr_i) \}$ with $z_k \in B(x_i, r_i)$.
This allows us to replace the third property in Lemma \ref{lemma-cover} by the following: For each ball $B(x_i, r_i)$,  either $c_0^{-1} \e \le r_i \le 6\e^{-1}_0 \e$ or 
$6 \e_0^{-1} \e < r_i < c$ and
\begin{equation}\label{re-7-00}
\wN (u_\e, y, r_i)
\le \ell - 1 +\delta_0 \quad 
\text{ for any } y \in B(x_i, r_i)\cap \mathcal{C}_\e (u_\e) \cap B(0, 1/2).
\end{equation}
Since $B(x_i, c)\cap B(0, 1/4)\neq \emptyset$ implies $B(x_i, c)\subset B(0, 1/2)$, we may remove $B(0, 1/2)$ in \eqref{re-7-00}.
This allows us to carry out an  induction argument on $\ell$.

Indeed, fix $i$ with $r_i> 6\e_0^{-1} \e$ and property \eqref{re-7-00}, and let $\phi (x) =u_\e (x_i +rx)$.
Then 
$$
\text{div} (B(x/(\e/r_i)) \nabla \phi ) =0 \quad \text{ in } B(0, 2),
$$
 where $B(y)= A(y+ (rx_i)/\e)$.
 Since $B \in \mathcal{A}(\lambda, \Gamma, M, \mu)$ and $(\e/r_i) < \e_0$, 
 we may apply Lemma \ref{lemma-cover} to the solution $\phi$ to obtain a cover for 
 $\mathcal{C} (\phi)\cap B(0, 1/4)$. As a result, we obtain a collection of balls $\{ B(x_{i, k}, r_{i, k} ) \}$ with the following properties:
 \begin{equation}
 \mathcal{C} (u_\e) \cap B(x_i, r_i/4)
 \subset \bigcup_k B(x_{i, k}, r_{i, k}/4),
 \end{equation}
 \begin{equation}
 \sum_k r_{i, k}^{d-2} \le C r_i^{d-2},
 \end{equation}
 and for each ball $B(x_{i, k}, r_{i, k})$, either $c_0 \e \le r_{i, k} \le 6 \e_0^{-1} \e$, or $r_{i,k} > 6\e_0^{-1} \e$ and
 \begin{equation}\label{re-7-01}
\wN (u_\e, y, r_{i, k} )
\le \ell - 2 +\delta_0 \quad 
\text{ for any } y \in B(x_{i, k} , r_{i, k} )\cap \mathcal{C}_\e (u_\e).
\end{equation}
We may continue the process 
 until we reach $\ell=1$. At each step, we keep those balls whose radii are between $c_0 \e $ and $6\e_0^{-1} \e$
and decompose those whose radii are greater than $6\e_0^{-1} \e$.
By  Theorem \ref{low-L},  if the right-hand side of the inequality \eqref{re-7-00} is  bounded by $3/2$, then
$$
\mathcal{C}(u_\e)\cap B(x_i, r_i/4)=\emptyset.
$$
As a result, 
we have constructed  a finite collection of balls $\{ B(z_i, s_i)\}$ with the properties that
\begin{equation}\label{final-1}
\mathcal{C}(u_\e)\cap B(0, 1/4) \subset \bigcup B(z_i, s_i/4),
\end{equation}
\begin{equation}\label{final-2}
\sum_i s_i^{d-2} \le C,
\end{equation}
and $c_0  \e \le  s_i \le 6 \e_0^{-1} \e$, where $\e_0=\e_0(L)$.

We point out the argument above works equally well if we replace $\e$ by any fixed $t \in (\e, c_0)$.
More precisely, the same argument also gives a collection of balls $\{ B(z_i, s_i)\}$  with the properties \eqref{final-1},
\eqref{final-2}, and $c_0t \le s_i \le C_0 t$ for all $i$.
To do this, in the place of \eqref{7-0-3},  one simply sets
\begin{equation}\label{7-0-3t}
r_t^* (x)= \sup \big\{ \e_0^{-1} t  < s\le 1/2: \wN (u_\e, x, s) \le  \ell -\delta _0 \big\}.
\end{equation}
If no such $s$ exists, i.e., $\wN(u_\e, x, s)>  \ell-\delta_0$ for all $\e_0^{-1} t < s\le 1/2$, define $r_t^* (x) =\e_0^{-1} t$. 
On the other hand, if $t\in (0, \e)$, we may use the local results in \cite{Naber-2017} on each ball $B(z_i, s_i)$ with $c_0\e \le s_i \le C_0 \e$,
 constructed for the case $t=\e$.
This gives $\{B(z_{i,k}, t) \}$ with the properties that
\begin{equation}\label{final-3}
\mathcal{C}(u_\e)\cap B(z_i, s_i /4) \subset \bigcup_k B(z_{i, k} , t/4),
\end{equation}
\begin{equation}\label{final-4}
\sum_k  t^{d-2} \le C s_i^{d-2}.
\end{equation}
In summary, we have proved that for any $t\in (0, c)$, the set $\mathcal{C} (u_\e)\cap B(0, 1/4)$ can be 
covered by a collection of balls $\{ B(z_i, Ct)\}$, and the number of the balls is bounded by $C t^{2-d}$, where $C$ depends on $L$.
This gives   the Minkowski  estimate \eqref{c-01}.
\end{proof}

We are now ready to give the proof of Theorem \ref{main-thm}.

\begin{proof}[\bf Proof of Theorem \ref{main-thm}]
Assume $A=A(y)$ satisfies the conditions \eqref{ellipticity}, \eqref{periodicity}, \eqref{smoothness} and \eqref{inv-0}.
Using the change of variables in Remark \ref{symm-re}, we may assume that $\widehat{A} +(\widehat{A})^T=2I$, and thus
$A\in \mathcal{A}(\lambda, \Gamma, M, \mu)$. Note that if $v_\e (x)=u_\e (S^{-1} x)$, as in Remark \ref{symm-re}, 
the doubling condition \eqref{doubling-0} is equivalent to 
$$
\fint_{B(0, 2)} v_\e^2 \le 4^N \fint_{B(0, 1)} v_\e^2.
$$

Let $u_\e\in H^1(B(0, 2))$ be a non-constant solution of $\mathcal{L}_\e (u_\e)=0$ in $B(0, 2)$.
Suppose that
\begin{equation}\label{d-7} 
\fint_{B(0, 2)} u_\e^2 \le 4^{N_0} \fint_{B(0, 1)} u_\e^2,
\end{equation}
for some $N_0>1$. It follows that
$$
\int_{B(0, 1)} |\nabla u_\e|^2 \le C \int_{B(0, 2)} u_\e^2 \le C \int_{B(0, 1)} u_\e^2
\le C \int_{\partial B(0, 1)} u_\e^2,
$$
where $C=C(N_0)$ depends on $N_0$. 
Thus,  $N(u_\e, 0, 1)\le C(N_0)$ and 
as a result, the case $\e\ge \e_0=\e_0(N_0)$ is covered by the results in \cite{Han-Lin-H-1998, Naber-2017}.
The constants may depend on $\e_0$, and
the periodicity condition is not needed.

To treat the case $0< \e< \e_0$, we note that \eqref{d-7} implies that for any $0< r< 1$,
\begin{equation}\label{f-100}
\fint_{B(0, r)} u_\e^2 \le C(N_0)  \fint_{B(0, r/2)} u_\e^2,
\end{equation}
by Theorem 1.2  in \cite{LS-Nodal}.
It follows that 
$$
\fint_{\partial B(x, 1)} u_\e^2 
\le C \fint_{B(0, 2)} u_\e ^2
\le C \fint_{B(0, 1/8)} u_\e^2
\le C \int_{B(x, 1/2)} u_\e^2
\le C \fint_{\partial B(x, 1/2)} u_\e^2,
$$
if $x\in B(0, 1/4)$, where $C$ depends on $N_0$.
It follows from Lemma \ref{lemma-c-1} that 
\begin{equation}\label{f-101}
|\big\{ x\in \R^d: \text{dist} (x, \mathcal{C}(u_\e)\cap B(0, 1/8)) < t \big\} |\le C(N_0) t^2.
\end{equation}
Finally, using harmonic approximation, one may show that
$$
\fint_{B(x, r)} u_\e^2 \le C(N_0) \fint_{B(x, r/2)} u_\e^2
$$
for any $x\in B(0, 3/4)$ and $0< r< r_0$, where $r_0$ is sufficiently small.
By a simple covering and scaling argument,
this allows us to replace $B(0, 1/8)$ in \eqref{f-101} by $B(0, 1/2)$  and completes the proof of Theorem \ref{main-thm}.
\end{proof}


 \bibliographystyle{amsplain}
 
\bibliography{Lin-Shen-2021.bbl}

\bigskip

\begin{flushleft}

Fanghua Lin, 
Courant Institute of Mathematical Sciences, 
251 Mercer Street, 
New York, NY 10012, USA.

Email: linf@cims.nyu.edu

\bigskip

Zhongwei Shen,
Department of Mathematics,
University of Kentucky,
Lexington, Kentucky 40506,
USA.

E-mail: zshen2@uky.edu
\end{flushleft}

\bigskip

\medskip

\end{document}